\documentclass{article}
\usepackage{latexsym}
\usepackage{amstext}
\usepackage{amsmath}
\usepackage{amssymb}
\usepackage{amsthm}
\usepackage{url}

\newtheorem{theorem}{Theorem}
\newtheorem{lemma}[theorem]{Lemma}
\newtheorem{corollary}[theorem]{Corollary}

\theoremstyle{definition}

\newtheorem{remark}[theorem]{Remark}

\def \kk {\overline{K}}

\def \bi {\mathbf{i}}

\def \co {\mathcal{O}}
\def \tors {\rm{tors}}

\DeclareMathOperator{\ord}{ord}
\DeclareMathOperator{\Supp}{supp}
\DeclareMathOperator{\opo}{o}
\DeclareMathOperator{\LCM}{LCM}

\DeclareMathOperator{\Cl}{Cl}
\DeclareMathOperator{\rk}{rk}
\DeclareMathOperator{\Jac}{Jac}

\DeclareMathOperator{\dv}{div}

\DeclareMathOperator{\pfree}{pfree}

\begin{document}
\bibliographystyle{amsplain}
\title{Variations on a theme of Runge:  effective determination of integral points on certain varieties}
\author{Aaron Levin}
\date{}
\maketitle
\begin{abstract}
We consider some variations on the classical method of Runge for effectively determining integral points on certain curves.  We first prove a version of Runge's theorem valid for higher-dimensional varieties, generalizing a uniform version of Runge's theorem due to Bombieri.  We then take up the study of how Runge's method may be expanded by taking advantage of certain coverings.  We prove both a result for arbitrary curves and a more explicit result for superelliptic curves.  As an application of our method, we completely solve certain equations involving squares in products of terms in an arithmetic progression.
\end{abstract}
\section{Introduction}
A fundamental problem in number theory is to determine the set of solutions over a number field $K$ (or its ring of integers $\co_K$) to a system of polynomial equations.  Equivalently, in more geometric terms, we are interested in determining the set of rational or integral points over $K$ on a variety $X$.  Despite the existence of powerful conjectures on this topic (e.g., Vojta's conjectures), in general, it can be said that very little is known for arbitrary varieties $X$.  When $X=C$ is a curve, however, the situation is much better, at least qualitatively.  For integral points, we have the classical theorem of Siegel which states that if an affine curve $C$ has infinitely many integral points, then $C$ must be rational and have at most two points at infinity.  When the genus of $C$ is at least two, Siegel's theorem is superseded by Faltings' celebrated result that such a curve has in fact only finitely many rational points over any number field.  Unfortunately, at present, both Siegel's and Faltings' theorems are ineffective.  That is, given a curve which is known to have only finitely many integral or rational points by Siegel's or Faltings' theorems, there is in general no known algorithm to provably find all of the finitely many integral or rational points on that curve.  

However, for certain classes of curves, over certain number fields $K$, there do exist effective techniques for finding all integral or rational points.  For instance, when the rank of the group of $K$-rational points in the Jacobian of $C$ is smaller than the genus of $C$, the Chabauty-Coleman method \cite{Col} frequently allows one to effectively determine $C(K)$.  For integral points, the most general effective techniques come from the theory of linear forms in logarithms \cite{Bak}.  Using this theory one can effectively determine, for example, the finitely many $S$-integral solutions to the superelliptic equation $y^n=f(x)$, where $n>1$, $f\in K[x]$ is $n$-th power free with at least three distinct roots, and $S$ is some finite set of places of $K$ containing the archimedean places.  There are, essentially, only a handful of such effective techniques, and so it is useful to expand the domain of applicability of any given method.  From this point of view, we will study the old method of Runge for effectively determining integral points on certain curves.

In 1887 Runge \cite{Run} proved the finiteness of the set of integral points on certain curves.  Although Runge did not state it, it is implicit in his proof that his method is effective.  In its most basic form, Runge proved:
\begin{theorem}[Runge]
\label{Run1}
Let $f\in \mathbb{Q}[x,y]$ be an absolutely irreducible polynomial of total degree $n$.  Let $f_0$ denote the leading form of $f$, i.e., the sum of the terms of total degree $n$ in $f$.  Suppose that $f_0$ factors as $f_0=g_0h_0$, where $g_0,h_0\in \mathbb{Q}[x,y]$ are nonconstant relatively prime polynomials.  Then the set of solutions to 
\begin{equation*}
f(x,y)=0, \quad x,y\in\mathbb{Z},
\end{equation*}
is finite and can be effectively determined.
\end{theorem}
Explicit bounds for the solutions in Runge's theorem (and its generalizations) have been given in \cite{Hil} and \cite{Wal}.  A geometric formulation of Runge's theorem which is valid for arbitrary rings of $S$-integers is the following.
\begin{theorem}
\label{Run2}
Let $C$ be a nonsingular projective curve defined over a number field $K$.  Let $\phi\in K(X)$ be a rational function on $C$.  Let $S$ be a finite set of places of $K$ containing the archimedean places.  Let $(\phi)_\infty$ be the divisor of poles of $\phi$ and let $r$ be the number of irreducible components over $K$ of the support of $(\phi)_\infty$.  If $r>|S|$ then the set $\{P\in C(K)\mid \phi(P)\in \co_{K,S}\}$ is finite and can be effectively determined.
\end{theorem} 
Theorem \ref{Run2} contains Theorem \ref{Run1} as a special case.  Indeed, under the hypotheses of Theorem \ref{Run1}, let $C$ be a projective closure of the affine plane curve defined by $f=0$ and let $\pi:C'\to C$ be a normalization.  Set $\phi=x\circ \pi$,  $K=\mathbb{Q}$, and $S=\{\infty\}$.  We can now apply Theorem~\ref{Run2}, noting that the factorization condition in Theorem~\ref{Run1} implies that the support of $(\phi)_\infty$ has at least two components over $\mathbb{Q}$.

Building on work of Sprind{\v{z}}uk \cite{Spr}, Bombieri \cite{Bom} proved a uniform version of Runge's theorem, allowing the number field $K$ and set of places $S$ to vary.  We state the theorem using the same notation as in Theorem \ref{Run2}.
\begin{theorem}[Bombieri, Sprind{\v{z}}uk]
\label{Bomb}
For $L\supset K$, let $r_L$ denote the number of irreducible components over $L$ of the support of $(\phi)_\infty$.  Then the set of points
\begin{equation*}
\bigcup_{\substack{L\supset K,S_L\\|S_L|<r_L}}\{P\in C(L)\mid \phi(P)\in \co_{L,S_L}\}
\end{equation*}
is finite and can be effectively determined.
\end{theorem}
Here $L$ ranges over all number fields and $S_L$ over sets of places of $L$ (containing the archimedean places).

The purpose of this paper is to expand the range of problems to which Runge's method can be applied and to give some explicit applications.  In the next section we prove a general version of Runge's theorem, extending Theorem~\ref{Bomb} to higher-dimensional varieties.  Following that, we show how unramified coverings of curves can be advantageously used in conjunction with Runge's method.  Roughly speaking, this allows Runge's method to be applied to curves which have large-rank rational torsion subgroups in their Jacobian and not too many places of bad reduction.  Natural examples of such curves are given by superelliptic curves $y^n=f(x)$, where $f$ splits into many factors over $\mathbb{Q}$ and the discriminant of $f$ has relatively few prime divisors.  We study such superelliptic curves in Section \ref{ssuper}.  Finally, as an application, we take up the well-studied problem of almost squares in products of arithmetic progressions and give some new results.

\section{Runge's theorem in higher dimensions}
Before stating our general formulation of Runge's theorem, we introduce some notation for integral points on arbitrary varieties.  Let $V$ be a variety (not necessarily projective or affine) defined over a number field $K$.  Let $S$ be a finite set of places of $K$ (containing, as throughout this paper, the archimedean places).  We call a set $R\subset V(K)$ a set of $S$-integral points on $V$ if for every regular function $\phi\in K(V)$ on $V$ there exists a nonzero constant $c\in K^*$ such that $c\phi(P)\in \co_{K,S}$ for all $P\in R$.  This definition is, in general, slightly more inclusive than the notion of $S$-integral points coming from Weil functions or integral models.  

It will be convenient to give definitions which also allow the set of places and the number field to vary.  We call a set $R\subset V(K)$ a set of $s$-integral points on $V$ if for every point $P\in R$ there exists a set of places $S_P$ of $K$ with $|S_P|\leq s$, and for every regular function $\phi\in K(V)$ on $V$ there exists a nonzero constant $c_\phi\in K^*$, independent of $P$, such that $c_\phi\phi(P)\in \co_{K,S_P}$.  Thus, essentially, an $s$-integral set of points on $V$ is a union of $S$-integral sets where $S$ varies over sets of places of $K$ with cardinality at most $s$.  Finally, if $s(L)$ is a function on number fields $L\supset K$, we call a set  $R\subset V(\kk)$ a set of $s(L)$-integral points on $V$ if for every point $P\in R$ there exists a set of places $S_P$ of $K(P)$ with $|S_P|\leq s(K(P))$, and for every regular function $\phi\in \kk(V)$ on $V$ there exists a nonzero constant $c_\phi\in K^*$, independent of $P$, such that $|c_\phi\phi(P)|_v\leq 1$ for all places $v$ of $K(P)$ not in $S$ (extending each place $v$ of $K(P)$ to $\kk$ in some fixed way).

In order to state our theorem, it will also be necessary to recall the definition of the Kodaira-Iitaka dimension $\kappa(D)$ of a divisor $D$.  Let $D$ be a divisor on a nonsingular projective variety $X$.  Then we let $L(D)=\{\phi\in \kk(X)\mid \dv(\phi)+D\geq 0\}$ and $h^0(D)=\dim H^0(X,\co(D))=\dim L(D)$.  If $h^0(nD)=0$ for all $n>0$ then we let $\kappa(D)=-\infty$.  Otherwise, we define the dimension of $D$ to be the integer $\kappa(D)$ such that there exist positive constants $c_1$ and $c_2$ with
\begin{equation*}
c_1 n^{\kappa(D)} \leq h^0(nD)\leq c_2 n^{\kappa(D)}
\end{equation*}
for all sufficiently divisible $n>0$.  We define a divisor $D$ on $X$ to be big if $\kappa(D)=\dim X$.

With the above notation, we generalize Bombieri's version of Runge's theorem to higher dimensions as follows.
\begin{theorem}
\label{gR}
Let $X$ be a nonsingular projective variety defined over a number field $K$.  Let $D=\sum_{i=1}^rD_i$ be an effective divisor on $X$ defined over $K$, with $D_1,\ldots, D_r$ distinct prime divisors (defined over $\kk$).  Suppose that the intersection of any $m+1$ of the supports of the divisors $D_i$ is empty.  Let $r(L)$ be the number of irreducible components of the support of $D$ over $L$.  Let $s(L)$ be a function such that $ms(L)<r(L)$.
\renewcommand{\theenumi}{(\alph{enumi})}
\renewcommand{\labelenumi}{\theenumi}
\begin{enumerate}
\item  If $\kappa(D_i)>0$ for all $i$, then any set $R$ of $s(L)$-integral points on $X\setminus D$ belongs to an effectively computable proper Zariski-closed subset $Z\subset X$.\label{R1}
\item If $D_i$ is big for all $i$, then there exists an effectively computable proper Zariski-closed subset $Z\subset X$ such that for any set $R$ of $s(L)$-integral points on $X\setminus D$, the set $R\backslash Z$ is finite (and effectively computable).\label{R2}
\item If $D_i$ is ample for all $i$, then all sets $R$ of $s(L)$-integral points on $X\setminus D$ are finite and effectively computable.\label{R3}
\end{enumerate}
\end{theorem}

We note that the hypothesis $ms(L)<r(L)$ in Theorem \ref{gR} is sharp, in that there are examples with $ms(L)=r(L)$ which violate the conclusions of parts \ref{R1}, \ref{R2}, or \ref{R3} of the theorem.  For instance, let $X=\mathbb{P}^2$ and let $D$ be a sum of $2m$ lines $D_1,\ldots,D_{2m}$, defined over $\mathbb{Q}$, with exactly $D_1,\ldots, D_m$ meeting at a point $P$ and $D_{m+1},\ldots,D_{2m}$ meeting at a different point $Q$.  Let $K$ be a number field with a set of places $S$, containing the archimedean place(s), of cardinality $|S|=2$.  Then the line through $P$ and $Q$ will contain an infinite set of $S$-integral points on $X\setminus D$.  It follows that a strict inequality $ms(L)<r(L)$ is necessary for part \ref{R3} to hold.

Given the geometric nature of the statement of our theorem, a few words are perhaps in order on what is meant here by ``effective".  Since our focus is on the arithmetic of varieties, we will take it as a given that one can explicitly compute certain fundamental geometric objects associated to the variety $X$ and the divisors $D_1,\ldots,D_r$.  For instance, we assume that we can compute explicit projective equations for the variety $X$ (and hence a presentation of the function field of $X$) and Riemann-Roch bases associated to the divisors $D_i$ and their linear combinations.  Alternatively, one could add the appropriate geometric data to the hypotheses of the theorem.  We also assume an effective version of the definition of a set of $s(L)$-integral points for the set $R$.  That is, given a regular function $\phi$ on $X\setminus D$, we assume that one can compute the constant $c\in k^*$ in the definition for the integral point set $R$.  Under the above assumptions, our proof gives, in principle, an algorithm for computing the projective equations of the set $Z$ (in parts \ref{R1} and \ref{R2}) or the set $R$ in part \ref{R3}.

Theorem \ref{gR} will be proved using Lemma \ref{W} below.  Lemma \ref{W} is a standard lemma which arises, for instance, in the construction of Weil functions.  However, in the interest of completeness, we provide a proof.  Before stating the lemma, we recall some relevant definitions.  We denote by $M_K$ the set of inequivalent absolute values of $K$.  We normalize our absolute values so that they extend the usual ones on $\mathbb{Q}$:  $|p|_v=\frac{1}{p}$ if $v$ corresponds to a prime ideal $\mathfrak{p}$ and $\mathfrak{p}|p$, and $|x|_v=|\sigma(x)|$ if $v$ corresponds to an embedding $\sigma:K\hookrightarrow\mathbb{C}$.    For $v\in M_K$, we denote by $K_v$ the completion of $K$ with respect to $v$.  We set $\|x\|_v=|x|_v^{[K_v:\mathbb{Q}_v]/[K:\mathbb{Q}]}$.  Thus, for $\alpha\in K$, the absolute multiplicative height is given by
\begin{equation*}
H(\alpha)=\prod_{v\in M_K}\max\{1,\|\alpha\|_v\}
\end{equation*}
and the absolute logarithmic height by $h(\alpha)=\log H(\alpha)$.  We define an $M_K$-constant to be a family of real numbers $(\gamma_v)_{v\in M_K}$ such that $\gamma_v=0$ for all but finitely many $v$.
\begin{lemma}
\label{W}
Let $X$ be a nonsingular projective variety defined over a number field $K$.  Extend each absolute value in $M_K$ in some way to $\kk$.  Let $\phi_1,\ldots, \phi_m\in K(X)$ be rational functions on $X$ without a common pole.  Then there exists an effectively computable $M_K$-constant $\gamma$, independent of the way each absolute value was extended, such that
\begin{equation*}
\min_{1\leq i\leq m}\log |\phi_i(P)|_v\leq \gamma_v
\end{equation*}
for all $v\in M_K$ and all $P\in X(\kk)$.
\end{lemma}
\begin{proof}
Fix an embedding of $X$ into projective space $\mathbb{P}^n$ such that $X$ is not contained in any hyperplane of $\mathbb{P}^n$.  For a point $P\in \mathbb{P}^n(\kk)$, let $(x_0(P),\ldots, x_n(P))$ be some set of projective coordinates for $P$.  Let 
\begin{equation*}
U_i=\{P\in X\mid x_i(P)\neq 0\},\quad i=0,\ldots, n.
\end{equation*}
  Let $(\phi_j)_0$ be the divisor of zeroes of $\phi_j$ and let $\mathcal{I}_j$ be the associated ideal sheaf on $X$.  Let $g_{i,j,1},\ldots, g_{i,j,k_{ij}}\in K[U_i]$ generate the sections of $\mathcal{I}_j$ over ${U_i}$.  Let $\phi_{i,j}=\phi_j|_{U_i}$.  Then for any $i,j,k$, we have $\frac{g_{i,j,k}}{\phi_{i,j}}\in K[U_i]$.  Furthermore, the functions $\frac{g_{i,j,k}}{\phi_{i,j}}$, $k=1,\ldots, k_{ij}$, have a common zero only at the poles of $\phi_{i,j}$.  Since $\phi_1,\ldots, \phi_m$ have no common pole, it follows that for any $i$, the functions $\frac{g_{i,j,k}}{\phi_{i,j}}$, $j=1,\ldots,m$, $k=1,\ldots,k_{ij}$, have no common zero on $U_i$.  By Hilbert's Nullstellensatz, there exist functions $h_{i,j,k}\in K[U_i]$ such that
\begin{equation}
\label{Nulleq}
\sum_{j=1}^m\sum_{k=1}^{k_{ij}}\frac{g_{i,j,k}}{\phi_{i,j}}h_{i,j,k}=1.
\end{equation}
Furthermore, and this is the key point regarding effectivity, Hilbert's Nullstellensatz can be made effective (e.g., \cite{MW}).  Let $F_{i,j,k}=g_{i,j,k}h_{i,j,k}$.  Let $f_{i,j}=\frac{x_j}{x_i}|_{U_i}$ be the functions on $U_i$ obtained by restriction of the rational functions $\frac{x_j}{x_i}$ on $\mathbb{P}^n$.  Then $f_{i,j}$, $j=0,\ldots, n$, generate $K[U_i]$.  It follows that each $F_{i,j,k}$ is a polynomial in the $f_{i,j}$.  Let
\begin{equation*}
E_{i,v}=\{P\in X(\kk)\mid |x_i(P)|_v=\max_j |x_j(P)|_v\},\quad i=0,\ldots, n.
\end{equation*}
Note that on $E_{i,v}$, $|f_{i,j}|_v\leq 1$.  Let $C_{i,j,k}$ be the number of terms of $F_{i,j,k}$ and let $|F_{i,j,k}|_v$ be the maximum absolute value with respect to $v$ of the coefficients of $F_{i,j,k}$ (as a polynomial in the $f_{i,j}$).  Let
\begin{align*}
\delta_v=\begin{cases}
1 &\text{if } v \text{ is archimedean}\\
0 &\text{if } v \text{ is nonarchimedean}.
\end{cases}
\end{align*}
Then $|F_{i,j,k}(P)|_v\leq C_{i,j,k}^{\delta_v}|F_{i,j,k}|_v$ for all $P\in  E_{i,v}$, $v\in M_K$.  It follows from \eqref{Nulleq} that for $P\in E_{i,v}$,
\begin{equation*}
\left(\sum_{j=1}^mk_{ij}\right)^{\delta_v}\max_{j,k} C_{i,j,k}^{\delta_v}|F_{i,j,k}|_v\max_j \left|\frac{1}{\phi_j}(P)\right|_v\geq 1,
\end{equation*}
or equivalently,
\begin{equation*}
\min_j \left|\phi_j(P)\right|_v\leq \left(\sum_{j=1}^mk_{ij}\right)^{\delta_v}\max_{j,k} C_{i,j,k}^{\delta_v}|F_{i,j,k}|_v.
\end{equation*}
Note that for any $v$, the sets $E_{i,v}$, $i=0,\ldots, n$, cover $X$ and that 
\begin{equation*}
\left(\sum_{j=1}^mk_{ij}\right)^{\delta_v}\max_{j,k} C_{i,j,k}^{\delta_v}|F_{i,j,k}|_v=1
\end{equation*}
for all but finitely many $v\in M_K$.  Thus, the lemma holds with the $M_K$-constant 
\begin{equation*}
\gamma_v=\log \max_i\left(\sum_{j=1}^mk_{ij}\right)^{\delta_v}\max_{j,k} C_{i,j,k}^{\delta_v}|F_{i,j,k}|_v.
\end{equation*}
\end{proof}

\begin{proof}[Proof of Theorem \ref{gR}]
We first prove part \ref{R1}.  Let $R$ be an $s(L)$-integral set of points on $X\setminus D$.  Let $L\supset K$ be a number field.  Let $D=\sum_{i=1}^{r(L)}E_i$ be the decomposition of $D$ into effective divisors over $L$.  Let $L'\subset L$ be the minimal field over which all the $E_i$ are defined.  Since $\kappa(D_i)>0$ for all $i$, and hence $\kappa(E_i)>0$ for all $i$, for $i=1,\ldots, r(L)$, there exists a non-constant rational function $\phi_i\in L'(X)$ such that the poles of $\phi_i$ lie in the support of $E_i$.  From the definition of $R$, after rescaling the $\phi_i$ (independent of $L$), we can assume that for any $i$ and any $P\in R$ with $K(P)=L$, $\phi_i(P)\in \co_{L,S}$ for some set of places $S$ of $L$ with $|S|\leq s(L)$.  Let $\mathcal{I}$ be the set of subsets $I\subset \{1,\ldots, r(L)\}$ such that the functions $\phi_i$, $i\in I$, have no common pole.  For $I\in \mathcal{I}$, let  $\gamma_I$ be the effectively computable $M_{L'}$-constant from Lemma \ref{W} for the set of functions $\phi_i$, $i\in I$.  Let $\gamma$ be the $M_{L'}$-constant defined by $\gamma_v=\max_{I\in \mathcal{I}} \gamma_{I,v}$.  Since the intersection of the supports of any $m+1$ distinct divisors $D_i$ is empty, any $m+1$ distinct functions $\phi_i$ have no common pole.  For $w\in M_L$, let $v_w$ denote the place of $L'$ lying below $w$.  It follows then from Lemma \ref{W} and the above definitions  that for any $P\in X(\bar{L})$ and $w\in M_L$, there exist at most $m$ functions $\phi_i$, $i\in \{1,\ldots, r(L)\}$, such that (extending $w$ to $\bar{L}$ in some way)
\begin{equation}
\label{ephi}
\log|\phi_i(P)|_w>\gamma_{v_w}.
\end{equation}
Let $P\in R$ with $K(P)=L$.  Then for all $i$, $\phi_i(P)\in \co_{L,S}$ for some set of places $S$ of $L$ with $|S|\leq s(L)$.  Since $r(L)>m|S|$, by the pigeon-hole principle and \eqref{ephi}, there exists some function $\phi=\phi_i$ such that $\log|\phi(P)|_w\leq \gamma_{v_w}$ for all $w\in S$.  As $\phi(P)\in \co_{L,S}$, it follows immediately that $h(\phi(P))\leq \sum_{w\in S}\max\{0,\gamma_{v_w}\}$.  Let $A$ be the maximum of the sum of $s(L)$ elements $\gamma_v$, $v\in M_{L'}$ (allowing repetitions).  Then $h(\phi(P))\leq A$.  Thus $P$ belongs to one of the finitely many proper Zariski-closed subsets of $X$ defined by $\phi_i=\alpha$, where $i\in \{1,\ldots, r(L)\}$ and $h(\alpha)\leq A$.  Note that the constant $A$ and the functions $\phi_i$ depended only on  the decomposition of $D$ over $L$ into the effective divisors $E_i$.  There are only finitely many possible such decompositions of $D$ into effective divisors.  Thus, going through the above proof over all such possible decompositions, we see that $R$ belongs to the union of finitely many effectively determinable proper Zariski-closed subsets of $X$.

The proofs of parts \ref{R2} and \ref{R3} are similar.  Let $L$, $L'$, and $E_i$ be as in the proof of part \ref{R1}.  Instead of considering functions $\phi_1,\ldots,\phi_{r(L)}$ with $\phi_i\in L(m_iE_i)$, for some $m_i>0$, we consider sets of functions that form bases of the spaces  $L(m_iE_i)$ for some sufficiently large $m_i$.  For instance, in the case where $D_i$, and hence $E_i$, is big for all $i$, let $m_i\in\mathbb{N}$ be such that the map $\Phi_{m_iE_i}$ associated to $L(m_iE_i)$ is birational outside of a proper Zariski-closed subset $Z_i\subset X$.  For each $i$, let $\phi_{i,1},\ldots, \phi_{i,l(m_iE_i)}\in L'(X)$ be a basis for $L(m_iE_i)$.  By scaling the functions, we can assume that they take on appropriately integral values as before.  Let $\mathcal{I}$ be the set of subsets $I\subset \{(i,j)\mid i\in\{1,\ldots,r(L)\},j\in\{1,\ldots,l(m_iE_i)\}\}$ such that the functions $\phi_{i,j}$, $(i,j)\in I$, have no common pole.  Let $\gamma$ be the $M_{L'}$-constant defined by $\gamma_v=\max_{I\in \mathcal{I}} \gamma_{I,v}$, where $\gamma_I, I\in \mathcal{I}$, is defined as in the proof of part \ref{R1}.  Define the constant $A$ with respect to $\gamma$ as before.  Let $P\in R$ with $K(P)=L$.  Then for all $i$ and $j$, $\phi_{i,j}(P)\in \co_{L,S}$ for some set of places $S$ of $L$ with $|S|\leq s(L)$.  Note that any $m+1$ functions $\phi_{i,j}$ with distinct $i$-indices have no pole in common.  Since $r(L)>m|S|$, using Lemma \ref{W} as before, there exists some $i$ such that for every function $\phi_{i,j}$, $j=1,\ldots,l(m_iE_i)$, we have $\log|\phi_{i,j}(P)|_w\leq\gamma_{v_w}$ for all $w\in S$.  It follows that the point
\begin{equation*}
\Phi_{m_iE_i}(P)=(\phi_{i,1}(P),\ldots, \phi_{i,l(m_iE_i)}(P))\in \mathbb{P}^{l(m_iE_i)-1}
\end{equation*}
is bounded in absolute logarithmic height by the constant $A$.  Since $\Phi_{m_iE_i}$ is birational outside of $Z_i$, it follows that all points $P\in R$ with $K(P)=L$ are contained in $\cup_{i=1}^r Z_i$ and a finite effectively determinable set of points.  As in part \ref{R1}, this set actually depends not on $L$, but on the divisors $E_i$.  As there are only finitely many possibilities for the $E_i$, \ref{R2} follows.  The proof for ample divisors \ref{R3} is essentially the same, except that in this case we can take $Z_i=\emptyset$ for all $i$.
\end{proof}

\section{Coverings and Runge's method}
In this section we take up the problem of expanding Runge's theorem for curves by taking advantage of unramified coverings.  Let $C$ be a curve defined over a number field $K$, $S$ a finite set of places of $K$, $\phi\in K(C)$, and $R=\{P\in C(K)\mid \phi(P)\in \co_{K,S}\}$.  It can happen that a straightforward application of Runge's method fails to prove finiteness for the set $R$, but that there is some unramified covering $\pi:X\to C$ such that Runge's method can be successfully applied to $X$, $\phi\circ \pi$, and $\pi^{-1}(R)$.  For instance, if $\phi$ has only a single pole, a straightforward application of Runge's method can never yield any information.  However, even in this case, it is sometimes possible to obtain nontrivial results by using coverings.  Roughly speaking, our reduction to a cover works if (1) there is a large-rank rational torsion subgroup in the Jacobian of $C$ and (2) if $C$ has relatively few places of bad reduction.  Examples of such curves are given in the subsequent sections, where we work out explicit bounds for integral points on certain families of curves. 

Before stating the main theorem, we introduce some more notation.  We let $\Jac(C)$ denote the Jacobian of $C$ and $\Jac(C)_{\rm{tors}}$ its torsion subgroup.  For a divisor $D$ on a variety $X$, we let $\Supp D$ denote the support of $D$.  For a finite abelian group $A$ and an integer $m>1$, we let $\rk_m A$, the $m$-rank of $A$, be the largest integer $r$ such that $(\mathbb{Z}/m\mathbb{Z})^r$ is a subgroup of $A$.  We denote the class group of the ring of $S$-integers of a number field $K$ by $\Cl(\co_{K,S})$.  If $L$ is a finite extension of the number field $K$ and $S$ a finite set of places of $K$,  we will use $\co_{L,S}$ to denote the ring of $T$-integers of $L$, where $T$ is the set of places of $L$ lying above places of $S$.  We let $\log_p$ denote the logarithm to the base $p$.

\begin{theorem}
\label{thmain2}
Let $C$ be a nonsingular projective curve of positive genus defined over a number field $K$.  Let $\phi\in K(C)$ be a rational function on $C$ such that every pole of $\phi$ is defined over $K$ and let $n_\infty$ denote the number of distinct poles of $\phi$.  Let $p$ be a (rational) prime.  Let $T$ be the union of the set of archimedean places of $K$, primes of $\co_K$ dividing $p$, and primes of bad reduction of $\Jac(C)$.    Then the set of integral points
\begin{equation*}
\bigcup_{\substack{L\supset K,S\\\log_p|S|+\rk_p \Cl(\co_{L,T})+\rk \co_{L,T}^*+1<\rk_p \Jac(C)(K)_{\rm{tors}}+\log_pn_\infty}}\{P\in C(L)\mid \phi(P)\in \co_{L,S}\}
\end{equation*}
is finite and can be effectively determined.
\end{theorem}
\begin{remark}
If a set of points $\cup_{L,S}\{P\in C(L)\mid \phi(P)\in \co_{L,S}\}$ is infinite, then Theorem \ref{thmain2} implies that for some $L$,
\begin{equation*}
\rk_p \Cl(\co_{L,T})\geq \rk_p \Jac(C)(K)_{\rm{tors}}+\log_pn_\infty-\log_p|S|-\rk \co_{L,T}^*-1.
\end{equation*}
Thus, Theorem \ref{thmain2} has the potential for being used to construct number fields with a large-rank ideal class group.  Indeed, this idea was pursued, in a slightly different form, in \cite{Lev} and \cite{Lev2}.  The lemmas in this section are essentially borrowed from \cite{Lev}.
\end{remark}
Before proving Theorem \ref{thmain2}, we prove some needed lemmas.  Let $r=\rk_p \Jac(C)(K)_{\rm{tors}}$.
\begin{lemma}
\label{fi}
Suppose that $C(K)\neq \emptyset$.  Then there exist $K$-rational divisors $D_1,\ldots, D_r$, whose divisor classes generate a subgroup $(\mathbb{Z}/p\mathbb{Z})^r\subset\Jac(C)(K)$ and rational functions $f_j\in K(C)$, $j=1,\ldots, r$, such that $\dv(f_j)=pD_j$ and for all $P\in C(\kk)$ such that $P$ is not a pole of $f_j$, 
\begin{equation}
\label{power}
f_j(P)\co_{K(P),T}=\mathfrak{a}_{P,j}^p
\end{equation}
for some fractional $\co_{K(P),T}$-ideal $\mathfrak{a}_{P,j}$.  If $P\in \Supp\dv(f_j)$, then there exists a nonzero function $h_j\in K(C)$ such that $P\not\in \Supp\dv(f_jh_j^p)$ and
\begin{equation}
f_jh_j^p(P)\co_{K(P),T}=\mathfrak{a}_{P,j}^p
\end{equation}
for some fractional $\co_{K(P),T}$-ideal $\mathfrak{a}_{P,j}$.
\end{lemma}
\begin{proof}
Let $c_1,\ldots, c_r$ be generators for a subgroup $(\mathbb{Z}/p\mathbb{Z})^r\subset\Jac(C)(K)$.  Let $Q\in C(K)$.  Let $\psi:C\hookrightarrow J=\Jac(C)$ be the $K$-rational embedding given by $P\mapsto [P-Q]$.  Let $\Theta=\psi(C)+\ldots+\psi(C)$ be the theta divisor on $J$.  Let $E_j=\Theta-t_{c_j}^*\Theta$, where $t_{c_j}$ denotes the translation-by-$c_j$ map on $J$.  By the theorem of the square, $pE_j$ is a principal divisor.  Let $f_j \in K(J)$ be such that $\dv(f_j)=pE_j$.  Since $[p]^*E_j\sim pE_j$ is principal, where $[p]$ is multiplication by $p$ on $J$, let $g_j\in K(J)$ be such that $\dv(g_j)=[p]^*E_j$.  It follows immediately that $f_j(px)=\alpha_jg_j(x)^p$ for some constant $\alpha_j\in K^*$.  Replacing $f_j$ by $f_j/\alpha_j$, we can assume that $f_j(px)=g_j(x)^p$.  Let $x,y\in J(\kk)$ with $py=x$.  It is a standard fact that the extension $K(y)/K(x)$ is unramified outside of (places lying above) $T$.  Since $f_j(x)=f_j(py)=g_j(y)^p$ and $K(y)/K(x)$ is unramified outside of $T$, if $x$ is not a pole of $f_j$ it follows that  $f_j(x)\co_{K(x),T}=\mathfrak{a}_j^p$ for some fractional $\co_{K(x),T}$-ideal $\mathfrak{a}_j$.  Consider $f_j|_{C}$, via the embedding $\psi:C\hookrightarrow J$.  Then in particular, for any $P\in C(\kk)$ with $P$ not a pole of $f_j|_{C}$, $f_j|_{C}(P)\co_{K(P),T}=\mathfrak{a}_j^p$ for some fractional $\co_{K(P),T}$-ideal $\mathfrak{a}_j$.  Let $D_j=\psi^*(\Theta-t_{c_j}^*\Theta)$.  Then $\dv(f_j|_{C})=\psi^*(pE_j)=p\psi^*(\Theta-t_{c_j}^*\Theta)=pD_j$.  We conclude the first part of the lemma by noting that $[\psi^*(\Theta-t_{c_j}^*\Theta)]=c_j$ \cite[Th.\ A.8.2.1]{SH}.  So $[D_j]=c_j$.

For the second part of the lemma, we note that by an elementary moving lemma, for any $j$ and any $P\in J(\kk)$, there exists a $K$-rational divisor $E_{j'}$ such that $E_{j'}\sim E_j$ and $P\not\in \Supp E_{j'}$.  Let $h_j\in K(J)$ with $(h_j)=E_{j'}-E_j$.  Then $P\not\in \Supp\dv(f_jh_j^p)$ and $f_j(x)h_j(x)^p=(g_j(y)h_j(x))^p$ on $J$, with $py=x$ as before.  So essentially the same proof as above gives the last assertion of the lemma.
\end{proof}

Let $f_1,\ldots, f_r$ be as in Lemma \ref{fi}.  Let $X$ be the nonsingular projective curve associated to the function field $K(C)\left(\sqrt[p]{f_1},\ldots,\sqrt[p]{f_r}\right)$ and let $\pi:X\to C$ be the natural morphism associated to this field extension of $K(C)$.  It is easy to see that $\deg \pi=p^r$.  This is equivalent to the following lemma.

\begin{lemma}
\label{LKum}
Let $D_1,\ldots, D_r$ be divisors whose divisor classes generate a subgroup $(\mathbb{Z}/p\mathbb{Z})^r\subset\Jac(C)$.  Let $f_1,\ldots, f_r\in K(C)$ be rational functions such that $\dv(f_j)=pD_j$ for all $j$.  Then 
\begin{equation*}
\left[\kk(C)\left(\sqrt[p]{f_1},\ldots,\sqrt[p]{f_r}\right):\kk(C)\right]=p^r.
\end{equation*}
\end{lemma}
\begin{proof}
By Kummer theory, this is equivalent to showing that $f_1,\ldots, f_r$ generate a subgroup of cardinality $p^r$ in $\kk(C)^*/\kk(C)^{p*}$.  Suppose that
\begin{equation}
\label{Kum} 
f_1^{i_1}f_2^{i_2}\cdots f_r^{i_r}=g^p
\end{equation}
for some $g\in \kk(C)^*$ and  integers $0\leq i_1,\ldots, i_r<p$.  Let $\dv(g)=E$, the principal divisor associated to $g$.  Then by \eqref{Kum}, $pE=\sum_{j=1}^r pi_jD_j$.  So $E=\sum_{j=1}^r i_jD_j$ is a principal divisor.  Since $[D_1],\ldots, [D_r]$ are independent $p$-torsion elements in $\Jac(C)$, it follows that $i_j=0$ for all $j$.
\end{proof}

\begin{lemma}
\label{ldeg}
Let $L\supset K$ be a number field.  Let $L'$ be the compositum of the number fields $K(Q)$, where $Q$ ranges over all points $Q\in X(\kk)$ with $\pi(Q)\in C(L)$.  Let $\zeta$ be a generator for the group of roots of unity in $L$.  Then
\begin{equation}
\label{degineq}
[L':L]\leq [L(\sqrt[p]{\zeta}):L]p^{\rk_p \Cl(\co_{L,T})+\rk \co_{L,T}^*}.
\end{equation}
\end{lemma}

\begin{proof}
We will work throughout with (fractional) $\co_{L,T}$-ideals.  Let $t=\rk_p \Cl(\co_{L,T})$.  Let $G=\left\{[\mathfrak{a}]\in\Cl(\co_{L,T})\mid [\mathfrak{a}]^p=1 \right\}$, a subgroup of $\Cl(\co_{L,T})$.  Then $G\cong (\mathbb{Z}/p\mathbb{Z})^{t}$.  Let $\mathfrak{b}_j$, $j=1,\ldots, t$, be ($\co_{L,T}$-)ideals whose ideal classes generate $G$.  Then for each $j$, $\mathfrak{b}_j^p=(\beta_j)$ for some $\beta_j\in L$.  Let $t'=\rk \co_{L,T}^*$.  Let $u_1,\ldots, u_{t'},\zeta$ be generators for $\co_{L,T}^*$.  Let $L'=L(\sqrt[p]{\beta_1},\ldots,\sqrt[p]{\beta_{t}},\sqrt[p]{u_1},\ldots,\sqrt[p]{u_{t'}},\sqrt[p]{\zeta})$.  Note that
\begin{equation*}
[L':L]\leq [L(\sqrt[p]{\zeta}):L]p^{t+t'}.
\end{equation*}
Let $Q\in X(\kk)$ with $P=\pi(Q)\in C(L)$.  We now show that $K(Q)\subset L'$.

First assume that $P$ is not a zero or pole of any $f_j$.  Then it follows from the definitions of $\pi$ and $X$ that $K(Q)=L(x_1,\ldots, x_r)$ for some choice of $x_j$ satisfying $x_j^p=f_j(P)$, $j=1,\ldots, r$.  We need to show that $x_j\in L'$ for all $j$.  By Lemma \ref{fi}, $(x_j^p)=(f_j(P))=\mathfrak{a}_j^p$ for some $\co_{L,T}$-ideal $\mathfrak{a}_j$.  Since $[\mathfrak{a}_j]\in G$, 
\begin{equation*}
\mathfrak{a}_j=(\alpha)\prod_{s=1}^{t}\mathfrak{b}_s^{c_s}
\end{equation*}
for some integers $c_s$ and some element $\alpha\in L$.  Therefore,
\begin{equation*}
(x_j^p)=\mathfrak{a}_j^p=(\alpha^p)\prod_{s=1}^{t}\left(\beta_s^{c_s}\right).
\end{equation*}
So $x_j^p=u\alpha^p\prod_{s=1}^{t}\beta_s^{c_s}$ for some unit $u\in \co_{L,T}^*$. Therefore,  $x_j=\alpha\sqrt[p]{u}\prod_{s=1}^{t}\sqrt[p]{\beta_s^{c_s}}$ for some choice of the $p$-th roots.  So $x_j\in L'$ for all $j$ as desired.

If $P$ is a zero or pole of some $f_j$ then the proof is similar, except we use the second part of Lemma \ref{fi} and the fact that $K(Q)=L(x_1,\ldots, x_r)$ for some choice of $x_j$ satisfying $x_j^p=f_jh_j^p(P)$, $j=1,\ldots, r$.
\end{proof}

\begin{proof}[Proof of Theorem \ref{thmain2}]
We apply Runge's method to the curve $X$ and the rational function $\phi\circ \pi$.  It is easily seen that $\pi$ is unramified.  Thus, since $\pi$ has $\deg\pi=p^r$, $\phi\circ \pi$ has $p^rn_\infty$ distinct poles over some number field $K'$.
By Theorem \ref{Bomb}, the set of points
\begin{equation*}
\bigcup_{\substack{L'\supset K',S\\|S|<p^rn_\infty}}\{Q\in X(L')\mid \phi(\pi(Q))\in \co_{L',S}\}
\end{equation*}
is finite and can be effectively determined.

Now consider $L\supset K$, $S$, and $P\in C(L)$ with $\phi(P)\in \co_{L,S}$.  Let $Q\in \pi^{-1}(P)$.  Then by Lemma \ref{ldeg} (using also that the poles of $\phi$ were all defined over $K\subset L$), 
\begin{equation*}
[K'(Q):L]\leq [L(\sqrt[p]{\zeta}):L]p^{\rk_p \Cl(\co_{L,T})+\rk \co_{L,T}^*}.
\end{equation*}
For $L'\supset L$, we let $S_{L'}$ be the set of places of $L'$ lying above places of $S$.  Trivially, we have $|S_{K'(Q)}|\leq |S|[K'(Q):L]$ and $\phi(\pi(Q))\in \co_{L,S}\subset \co_{K'(Q),S}$.  So combining the above, we see that the set of points
\begin{equation*}
\bigcup_{\substack{L\supset K,S\\|S|[L(\sqrt[p]{\zeta}):L]p^{\rk_p \Cl(\co_{L,T})+\rk \co_{L,T}^*}<p^rn_\infty}}\{P\in C(L)\mid \phi(P)\in \co_{L,S}\}
\end{equation*}
is finite and can be effectively determined.
\end{proof}

\section{Superelliptic curves}
\label{ssuper}

The extension of Runge's theorem proven in the last section (Theorem \ref{thmain2}) is only useful for curves $C$ which have a large-rank rational torsion subgroup in $\Jac(C)$.  A natural class of curves which have this property are superelliptic curves $C$ defined by
\begin{equation*}
y^p=f(x)=\prod_{i=1}^rf_i(x), \quad f_1,\ldots,f_r\in K[x],
\end{equation*}
where $f_1,\ldots, f_r$ are $p$-th power free, pairwise coprime, nonconstant polynomials.
Indeed, if $p\nmid \deg f$, it is easily shown that $\rk_p \Jac(C)(K)_{\rm{tors}}\geq r-1$.  An explicit version of Theorem \ref{thmain2} for superelliptic curves is given in the following result.  For a polynomial $f$, we let $D_f$ be the discriminant of $f$, $|f|_v$ be the maximum absolute value of the coefficients of $f$ with respect to $v$, and $H(f)$ be the absolute multiplicative height of the coefficients of $f$ as a point in projective space.  For a number field $K$, we let $D_K$ be the absolute discriminant of $K$.  For an element $\alpha\in K^*$, we let $\omega_K(\alpha)$ denote the number of distinct prime ideals in the ideal factorization of $\alpha\co_K$ (setting $\omega(n)=\omega_\mathbb{Q}(n))$.  For a set of places $S$ of $K$, we let $S_\infty$ denote the set of archimedean places in $S$.
\begin{theorem}
\label{sr}
Let $p$ be a prime.  Let $f_1,\ldots, f_r\in \co_K[x]$ be monic, pairwise coprime, nonconstant polynomials.  Suppose that each $f_i$, $i=1,\ldots, r$, is an $n$th power, $(n,p)=1$, $n\geq 1$, of a polynomial $g_i$ which has no repeated roots.  For $a\neq 0$, let $C_a$ be the superelliptic affine plane curve defined by
\begin{equation*}
ay^p=f(x)=\prod_{i=1}^rf_i(x).
\end{equation*}
Let $d=\deg f$.  Let $\zeta_p$ be a primitive $p$-th root of unity.  Let $\epsilon=0$ if $p$ divides $\deg f_i$ for $i=1,\ldots, r$, $\epsilon=1$ if $(d,p)=1$, and $\epsilon=2$ if $p$ divides $d$ but $p$ doesn't divide $\deg f_i$ for some $i$.  Then the set of integral points
\begin{equation*}
R=\bigcup_{\substack{K'\supset K,S,a\in \co_{K',S}\\\omega_{K'}(aD_f)+|S_{\infty}|+\log_p |S|+\rk_p \Cl(\co_{K'})+\epsilon<r}}\{P\in C_a(\co_{K',S})\}
\end{equation*}
is finite.  Let $d'=\deg \prod_{i=1}^rg_i$.  Specifically, for each point $P\in R$, we have the bound
\begin{equation}
\label{HB}
H(x(P))\leq \left|D_{K(\zeta_p)}\right|^{d'p^{2r}}\left(4d'p^3\prod_{i=1}^r H(g_i)\right)^{d'^2p^{3r}}.
\end{equation}
\end{theorem}
\begin{remark}
The bound \eqref{HB} is a somewhat crude estimate.  For $P\in C_a(\co_{K',S})$ such that $r$ is large compared to $\omega_{K'}(aD_f)+|S_{\infty}|+\log_p |S|+\rk_p \Cl(\co_{K'})$, the inequality \eqref{HB} can be improved.  Such improved bounds are easily calculated from Theorem \ref{mainth} and an examination of the proof of Theorem \ref{sr}. 
\end{remark}
In many cases the bounds in Theorem \ref{sr} are small enough that, combined with other techniques, it is practical to find all $S$-integral points on certain superelliptic curves.  This is illustrated in Section \ref{sper}, where we compute certain sets of $S$-integral points on some curves related to the problem of squares in products of terms in an arithmetic progression.

\begin{theorem}
\label{wth}
Let $p$ be a prime.  Let $f_1,\ldots, f_r\in \mathbb{Z}[x]$ be monic, pairwise coprime, nonconstant polynomials with no repeated roots.  Let $d=\deg \prod_{i=1}^rf_i$ and $H=\max_i H(f_i)$.  Let $\pfree(n)$ denote the $p$-th-power-free part of an integer $n$.  Then for $x\in \mathbb{Z}$,
\begin{equation}
\label{weq}
\omega\left(\prod_{i=1}^r \pfree(f_i(x))\right)\geq r-1
\end{equation}
if
\begin{equation*}
|x|>(8dp^3H)^{2d^2p^{r+1}}.
\end{equation*}
If $p\mid \deg f_i$, $i=1,\ldots, r$, then the same statement holds with \eqref{weq} replaced by
\begin{equation}
\label{weq2}
\omega\left(\prod_{i=1}^r \pfree(f_i(x))\right)\geq r.
\end{equation}
\end{theorem}
Again, a better, but more complicated estimate in Theorem \ref{wth} follows from the proof and Theorem \ref{mainth}.
\begin{remark}
Inequalities \eqref{weq} and \eqref{weq2} are likely to be sharp for infinitely many values of $x$, at least for certain sets of polynomials.  Indeed, if Schinzel's ``Hypothesis $H$" \cite{Sch} is true, most sets of polynomials $\{f_2,\ldots, f_r\}$ will take prime values simultaneously at infinitely many points $x=n^p$.  Thus, if we take $f_1=x$, at such points the bound in inequality \eqref{weq} will be achieved.  Similarly, in general, the bound in inequality \eqref{weq2} should also be attained infinitely often.
\end{remark}

Theorems \ref{sr} and \ref{wth}, whose proofs we postpone until the end of the section, will be a consequence of the following general result.
\begin{theorem}
\label{mainth}
Let $p$ be a prime and $\zeta_p$ a primitive $p$-th root of unity.  Let $K$ be a number field.  Let $f_1,\ldots, f_r\in \co_K[x]$ be monic, pairwise coprime, nonconstant polynomials with no repeated roots.  Let $K'\supset K$ be a number field and $S$ a finite set of places of $K'$ containing the archimedean places.  Let $\alpha\in \co_{K',S}$.  Let $\beta_i=\sqrt[p]{f_i(\alpha)}$, $i=1,\ldots, r$.  Let $L=K'(\beta_1,\ldots,\beta_r,\zeta_p)$.  Let $T$ be the set of places of $L$ lying above places of $S$ and let $t=|T|$.   Let $S_\infty$ denote the set of archimedean places of $S$.  Let $d_i=\deg f_i$ and $d=\sum_{i=1}^rd_i$.  Let
\begin{align*}
H_v&=\max_i |f_i|_v, \quad v\in S_\infty,\\
B&=\prod_{v\in S_\infty} \left(H_v+1\right)^{[K'_v:\mathbb{Q}_v]/[K':\mathbb{Q}]},\\
\delta&= \left\lceil \frac{p^{r-1}((p-1)(d-1)-2)+t+2}{p^{r-1}-t}\right\rceil,\\
\delta'&=\left\lceil \frac{p^{r-1}((p-1)(d-2)-2)+t+2}{p^r-t}\right\rceil,\\
m&=p^{r-1}\left(\delta+1-\frac{(p-1)(d-1)}{2}\right),\\
m'&=p^{r-1}\left(p\delta'+1-\frac{(p-1)(d-2)}{2}\right).
\end{align*}
If $p\nmid d_i$ for some $i$ and $t<p^{r-1}$, then
\begin{align}
\label{h1}
H(\alpha)&\leq \left|D_{K(\zeta_p)}\right|^{mp/(2[K(\zeta_p):\mathbb{Q}])} 2^{(\delta+1)(\delta t+1)+2p}p^{(\delta+1)(t(2\delta-p)+2)}m^{p((\delta+1)t/2+1)}B^{\delta(\delta+1)t+2\delta+1}\\
&< \left|D_{K(\zeta_p)}\right|^{dp^{2r}}(2dp^3B)^{d^2p^{3r}}.\notag
\end{align}
If $p\mid d_i$ for $i=1,\ldots, r$ and $t<p^{r}$, then
\begin{align}
\label{h2}
H(\alpha)&\leq \left|D_{K(\zeta_p)}\right|^{m'/(2[K(\zeta_p):\mathbb{Q}])} 2^{(\delta'+1)(\delta' t+1)+1}p^{(\delta'+1)(t(2\delta'-1)+2)}m'^{(\delta'+1)t/2+1}B^{\delta'(\delta'+1)t+2\delta'+1}\\
&< \left|D_{K(\zeta_p)}\right|^{dp^{2r}}(2dp^3B)^{d^2p^{3r}}.\notag
\end{align}
\end{theorem}
Some of the arguments in our proof, particularly those involving Puiseux series, follow arguments given in \cite{Hil} and \cite{Wal}.
\begin{proof}
We will apply Runge's method to the curve $C\subset \mathbb{A}^{r+1}$ defined by $y_i^p=f_i(x)$, $i=1,\ldots, r$.  We first collect some geometric facts about the curve $C$, including the rational functions on $C$ that we will be interested in.

Consider the function field of $C$, $K(C)=K(x,y_1,\ldots,y_r)$.  It follows from the fact that $f_1,\ldots, f_r$ are $p$-th power free, pairwise coprime polynomials that $[K(C):K(x)]=p^r$ and a basis for $K(C)$ over $K(x)$ is given by the elements $y_1^{i_1}\cdots y_r^{i_r}$, $0\leq i_1,\ldots, i_r\leq p-1$.  Let
\begin{align*}
\mathcal{M}(\delta)&=\left\{x^{i_0}y_1^{i_1}\cdots y_r^{i_r}\mid (i_0,i_1,\ldots, i_r)\in \mathbb{N}\times \{0,\ldots,p-1\}^r, \sum_{k=0}^ri_kd_k \leq \delta\right\},\\
m(\delta)&=\#\mathcal{M}(\delta),
\end{align*}
where we have also set $d_0=p$.  It follows easily from the definitions that the generating function for $m(i)$, $i\in \mathbb{N}$, is given by 
\begin{equation*}
\sum_{i=0}^\infty m(i)x^i=\frac{\prod_{i=1}^r\sum_{j=0}^{p-1}x^{jd_i}}{(1-x)(1-x^p)}.
\end{equation*}
We first prove the theorem in the case that for some $i_0$, $p\nmid d_{i_0}$.  Under this assumption, $\sum_{j=0}^{p-1}x^{j}$ divides $\sum_{j=0}^{p-1}x^{jd_{i_0}}$.  So
\begin{equation*}
\frac{\prod_{i=1}^r\sum_{j=0}^{p-1}x^{jd_i}}{(1-x)(1-x^p)}=\frac{1}{(1-x)^2}\frac{\prod_{i=1}^r\sum_{j=0}^{p-1}x^{jd_i}}{\sum_{j=0}^{p-1}x^{j}}=\frac{g(x)}{(1-x)^2},
\end{equation*}
where $g(x)$ is a polynomial of degree $(p-1)\sum_{i=1}^rd_i-(p-1)=(p-1)(d-1)$.  So
\begin{equation*}
\frac{\prod_{i=1}^r\sum_{j=0}^{p-1}x^{jd_i}}{(1-x)(1-x^p)}=\frac{g(1)}{(1-x)^2}-\frac{g'(1)}{1-x}+h(x),
\end{equation*}
where $\deg h=(p-1)(d-1)-2$.  It is easily calculated that $g(1)=p^{r-1}$ and $g'(1)=\frac{p^{r-1}(p-1)(d-1)}{2}$.  Thus, it follows that for $\delta\geq (p-1)(d-1)-1$, 
\begin{equation*}
m(\delta)=(\delta+1)p^{r-1}-\frac{p^{r-1}(p-1)(d-1)}{2}=p^{r-1}\left(\delta+1-\frac{(p-1)(d-1)}{2}\right).
\end{equation*}

Let $C'$ be a projective closure of $C$ in $\mathbb{P}^{r+1}$ and let $\pi:\tilde{C}\to C'$ be the normalization of $C'$.  Since $y_{i_0}^p=f_{i_0}(x)$, we have $p\cdot(y_{i_0})_{\infty}=d_{i_0}\cdot(x)_{\infty}$ (viewing $x$ and $y_{i_0}$ as rational functions on $\tilde{C}$).  As $p\nmid d_{i_0}$, it follows that $(x)_{\infty}=pD_\infty$ for some effective divisor $D_{\infty}$.  

\begin{lemma}
\label{gl}
The genus of $\tilde{C}$ is
\begin{equation}
\label{geneq}
g(\tilde{C})=p^{r-1}\left(\frac{(p-1)(d-1)}{2}-1\right)+1.
\end{equation}
For $\delta\geq (p-1)(d-1)-1$, the elements of $\mathcal{M}(\delta)$ form a basis for $L(\delta D_\infty)$.  Furthermore, $D_\infty$ is the sum of $p^{r-1}$ distinct points of $\tilde{C}$.
\end{lemma}
\begin{proof}
Consider the morphism $\phi:\tilde{C}\to \mathbb{P}^1$ obtained from the rational function $x\circ \pi$ on $\tilde{C}$.  Note that $\deg \phi=p^r$.  If $x(\pi(Q))$ is a root of $f_i$ for some $i$, then $\phi$ has ramification index $p$ at $Q$.  This gives $dp^{r-1}$ points of $\tilde{C}$ with ramification index $p$ with respect to $\phi$.  Since $(x)_{\infty}=pD_\infty$, every point of $\tilde{C}$ above $\infty\in \mathbb{P}^1$ has ramification index divisible by $p$.  Alternatively, we can see this as follows.  Let $\tilde{Y}$ be the nonsingular projective model of the affine plane curve $Y$ defined by $y^p=f_{i_0}(x)$.  Since $p\nmid d_{i_0}$, it is easily seen that $Y$ has a unique point $Q$ at infinity, and the map $\tilde{Y}\to\mathbb{P}^1$ induced by the projection map $(x,y)\mapsto x$ has ramification index $p$ at $Q$.  Since the map $\phi$ factors through this map, it follows that every point above $\infty\in \mathbb{P}^1$ on $\tilde{C}$ has ramification index divisible by $p$.  Let $n$ be the number of distinct points in $D_\infty$.  Note that $n\leq p^{r-1}$.  Then by the Riemann-Hurwitz formula applied to $\phi$, we conclude that
\begin{equation*}
2g(\tilde{C})-2\geq -2p^r+d(p-1)p^{r-1}+p^r-n,
\end{equation*}
or $g(\tilde{C})\geq \frac{p^{r-1}(dp-d-p)-n}{2}+1\geq p^{r-1}\left(\frac{(p-1)(d-1)}{2}-1\right)+1$.  

On the other hand, since $(x)_\infty=pD_\infty$ and $(y_i)_\infty=d_iD_\infty$, we see that $\mathcal{M}(\delta)\subset L(\delta D_\infty)$.  Since the elements of $\mathcal{M}(\delta)$ are linearly independent, for $\delta\geq (p-1)(d-1)-1$ we have $\dim L(\delta D_\infty)=l(\delta D_\infty)\geq m(\delta)=p^{r-1}\left(\delta+1-\frac{(p-1)(d-1)}{2}\right)$.  Note also that $\deg D_\infty=p^{r-1}$.  Thus, by Riemann-Roch, for $\delta p^{r-1}\geq 2g(\tilde{C})-1$, we have $l(\delta D_\infty)=\delta p^{r-1}+1-g(\tilde{C})$.  It follows that $g(\tilde{C})\leq p^{r-1}\left(\frac{(p-1)(d-1)}{2}-1\right)+1$.  This proves \eqref{geneq}.  Thus, for $\delta\geq (p-1)(d-1)-1$, $l(\delta D_\infty)=m(\delta)$, and the elements of $\mathcal{M}(\delta)$ form a basis for $L(\delta D_\infty)$.  Additionally, we see that we must have $n=p^{r-1}$, and so $D_\infty$ is the sum of exactly $p^{r-1}$ distinct points of $\tilde{C}$.
\end{proof}

We now work out some facts about the Puiseux expansions of the algebraic functions $y_i$.  For each $i$, $y_i^p-f_i(x)$ can be factored using Puiseux series as 
\begin{equation}
\label{fact}
y_i^p-f_i(x)=\prod_{j=0}^{p-1}y_i-y_{i,j}(x).  
\end{equation}
Explicitly, if $f_i=x^{d_i}+\sum_{j=0}^{d_i-1}a_{i,j}x^j$, then writing 
\begin{equation*}
f_i^{\frac{1}{p}}=x^{\frac{d_i}{p}}\left(1+\sum_{j=0}^{d_i-1}a_{i,j}x^{j-d_i}\right)^{\frac{1}{p}}
\end{equation*}
and using the Taylor series for $z=(1+t)^{\frac{1}{p}}$ about $z_0=1$, we see that (after possibly reindexing),
\begin{equation}
\label{Pseries}
y_{i,j}=\sum_{k=0}^{\infty}c_{i,j,k}x^{\frac{d_i}{p}-k}=\zeta_p^jx^{\frac{d_i}{p}}\sum_{k=0}^{\infty}\binom{1/p}{k} \left(\sum_{l=0}^{d_i-1}a_{i,l}x^{l-d_i}\right)^k, \quad j=0,\ldots, p-1.
\end{equation}
Expanding out the right-hand side of \eqref{Pseries} appropriately then explicitly gives the coefficients $c_{i,j,k}$ and the $p$ Puiseux expansions $y_{i,j}(x)$,  $j=0,\ldots, p-1$.  When evaluating the Puiseux series \eqref{Pseries} at a point $x$, we assume the choice of some fixed branch of $x^{\frac{1}{p}}$.  

\begin{lemma}
\label{lconv}
Let $y_{i,j}(x)$ be the Puiseux expansions in \eqref{Pseries}.  Then
\begin{equation}
\label{denom}
p^{2k-1}c_{i,j,k}\in \co_K[\zeta_p], \quad k\geq 1.
\end{equation}
Extend each place $v\in M_L$ to $\bar{L}$ in some way.  For $v\in M_L$ and $x\in \bar{L}$, $y_{i,j}(x)$ converges $v$-adically if
\begin{align*}
|x|_v&\geq |f_i|_v+1, && \text{if $v$ is archimedean},\\
|x|_v&>\frac{1}{|p|_v^2}, && \text{if $v$ is nonarchimedean}.
\end{align*}
\end{lemma}
\begin{proof}
Each coefficient $c_{i,j,k}$ is an $\co_K[\zeta_p]$-integral linear combination of numbers $\binom{1/p}{l}$, $0\leq l\leq k$.  Since $p^{2l-1}\binom{1/p}{l}\in \mathbb{Z}$ for all $l\geq 1$, \eqref{denom} follows.  We now prove the assertions about convergence.  If $v$ is archimedean and $|x|_v\geq |f_i|_v+1$, then we note that 
\begin{equation}
\label{est}
\left|\sum_{j=0}^{d_i-1}a_{i,j}x^{j-d_i}\right|_v\leq |f_i|_v \sum_{j=0}^{d_i-1}|x|_v^{j-d_i}<\frac{|f_i|_v}{|x|_v-1}\leq 1.
\end{equation}
It now follows easily from \eqref{Pseries} that $y_{i,j}(x)$ converges.  For nonarchimedean $v$, $y_{i,j}(x)$ converges if and only if $\lim_{k\to \infty}|c_{i,j,k}x^{\frac{d_i}{p}-k}|_v=0$.  By \eqref{denom}, $|c_{i,j,k}|_v\leq \frac{1}{|p|_v^{2k-1}}$.  Thus, in this case, it is clear that $y_{i,j}(x)$ converges $v$-adically if $|x|_v>\frac{1}{|p|_v^2}$. 
\end{proof}

Let $T'\subset T$ be the set of places $v\in T$ such that
\begin{align*}
|\alpha|_v&> H_v+1, && \text{if $v$ is archimedean},\\
|\alpha|_v&>\frac{1}{|p|_v^2}, && \text{if $v$ is nonarchimedean}.
\end{align*}

By \eqref{fact} and Lemma \ref{lconv}, we have that for any $i\in \{1,\ldots,r\}$ and any $v\in T'$, there exists $j_{i,v}\in \{0,\ldots,p-1\}$ such that $y_{i,j_{i,v}}(\alpha)$ converges $v$-adically to $\beta_i$.  As $|x|_v\to \infty$, the point on $C$ defined by $y_i=y_{i,j_{i,v}}(x)$, $i=1,\ldots, r$, converges $v$-adically to some point $P_v$ at infinity.  More precisely, using the map $\pi$ to pull back points to $\tilde{C}$, we can view $P_v\in \Supp D_\infty\subset \tilde{C}$.  By Lemma~\ref{gl}, since $\#(\Supp D_\infty)=p^{r-1}$, there are $p^{r-1}$ possibilities for the point $P_v$.  Let $E$ be the divisor on $\tilde{C}$ given by $E=\sum_{P\in\{P_v\mid v\in T'\}}P$.    We want to find functions $g\in L(\delta D_\infty)$ such that $g$ vanishes at every point $P_v$, $v\in T'$.  In other words, we want functions $g\in L(\delta (D_\infty-E)-E)$.  By assumption, $\deg E\leq |T'|\leq |T|<p^{r-1}=\deg D_\infty$, and so $L(\delta (D_\infty-E)-E)\neq 0$ for $\delta\gg 0$.  A function $g$ vanishing at all points $P_v$, $v\in T'$, will be $v$-adically small at the point $(x,y_1,\ldots,y_r)=(\alpha,\beta_1,\ldots,\beta_r)$ for $v\in T'$.  More precisely:

\begin{lemma}
\label{mainl}
Let $\delta\geq (p-1)(d-1)-1$ be a positive integer such that $(\delta+1)\deg E<m(\delta)=l(\delta D_\infty)$.  Let $N=l(\delta (D_\infty-E)-E)\geq m(\delta)-(\delta+1)\deg E$.  Then there exist polynomials $g_1,\ldots, g_N\in \co_K[\zeta_p][x_0,\ldots, x_r]$ such that $g_1(x,y_1,\ldots,y_r),\ldots, g_N(x,y_1,\ldots,y_r)$ form a basis for $L(\delta (D_\infty-E)-E)$,
\begin{equation}
\label{gh}
H(g_i)\leq \left|D_{K(\zeta_p)}\right|^{\frac{N}{2[K(\zeta_p):\mathbb{Q}]}}\left(\frac{\sqrt{m(\delta)}(2p^2B)^{\delta/p}}{p}\right)^{m(\delta)-N}, \quad i=1,\ldots, N,
\end{equation}
and for all $v\in T', i=1,\ldots, N$,
\begin{equation}
\label{gv}
\left|g_i\left(\alpha,\beta_1,\ldots,\beta_r\right)\right|_v\leq
\begin{cases}
\frac{m(\delta)|g_i|_v(2H_v+2)^{(\delta+1)/p}}{2^{1/p}|\alpha|_v^{1/p}\left(1- \frac{H_v+1}{|\alpha|_v}\right)} &\text{if } v \text{ is archimedean},\\
\frac{|g_i|_v}{|\alpha|_v^{1/p}|p|_v^{2(\delta+1)/p-1}}
 &\text{if } v \text{ is nonarchimedean}.
\end{cases}
\end{equation}
\end{lemma}
\begin{proof}
We construct the basis $g_1,\ldots, g_N$ by looking at the Puiseux expansions of functions in $L(\delta D_\infty)$.  A nonzero polynomial $g(x,y_1,\ldots, y_r)$ vanishes at $P_v$ if and only if in the Puiseux expansion $g\left(x,y_{1,j_{1,v}}(x),\ldots,y_{r,j_{r,v}}(x)\right)$, we have
\begin{equation*}
\ord_x g\left(x,y_{1,j_{1,v}}(x),\ldots,y_{r,j_{r,v}}(x)\right)<0.
\end{equation*}
Let $v_1,\ldots, v_e$ be a minimal set of places in $T'$ such that $\{P_{v_1},\ldots, P_{v_e}\}=\{P_v\mid v\in T'\}$.  Since, by Lemma \ref{gl}, $\mathcal{M}(\delta)$ is a basis for $L(\delta D_\infty)$, explicitly determining the vector space $L(\delta (D_\infty-E)-E)\subset L(\delta D_\infty)$ is equivalent to solving the system of equations:
\begin{equation}
\label{sys}
\ord_x\sum_{\bi=(i_0,\ldots,i_r) \in I(\delta)}c_\bi x^{i_0}y_{1,j_{1,v_l}}(x)^{i_1}\cdots y_{r,j_{r,v_l}}(x)^{i_r}<0,\quad l=1,\ldots, e,
\end{equation}
where $I(\delta)=\{(i_0,\ldots, i_r)\in \mathbb{N}^{r+1}\mid x^{i_0}y_1^{i_1}\cdots y_r^{i_r}\in \mathcal{M}(\delta)\}$.  The only monomials which appear on the left-hand side of \eqref{sys} with nonnegative degree are $x^{i/p}$, $i=0,\ldots, \delta$.  So \eqref{sys} yields a system of $(\delta+1)e$ equations in $m(\delta)$ variables.  Let $A$ be the corresponding $(\delta+1)e\times m(\delta)$ matrix.  We now bound the height of the matrix $A$.

Let $(i_0,\ldots,i_r)\in I(\delta)$.  Then, by \eqref{Pseries}, we have
\begin{equation}
\label{Pseries2}
x^{i_0}y_{1,j_{1,v}}(x)^{i_1}\cdots y_{r,j_{r,v}}(x)^{i_r}=\zeta x^{\sum_{l=0}^r i_ld_l/p}\sum_{k=0}^\infty a_kx^{-k},
\end{equation}
for some $p$-th root of unity $\zeta$ and $a_k\in K$, $k\in \mathbb{N}$.  Explicitly, $a_k$ can be computed from the Taylor series for $\left(\prod_{j=1}^r x^{i_jd_j}f_j(1/x)^{i_j}\right)^{1/p}$ at $x=0$:
\begin{equation*}
F(x)=\left(\prod_{j=1}^r x^{i_jd_j}f_j(1/x)^{i_j} \right)^{1/p}=\sum_{k=0}^\infty a_kx^k,
\end{equation*}
with $a_0=1$.  As before, this immediately implies that 
\begin{equation}
\label{ad}
p^{2k-1}a_k\in \co_K, \quad k\geq 1.
\end{equation}
We now estimate the size of $a_k$.  Let $v\in T_\infty$, where $T_\infty$ is the set of archimedean places in $T$.  Since $F(z)$ is analytic in $\mathbb{C}_v$ for $|z|_v\leq \min_i \frac{1}{|f_i|_v+1}=\frac{1}{H_v+1}$, we have
\begin{equation*}
|a_k|_v=\left|\frac{1}{2\pi i}\int_{|z|_v=\frac{1}{H_v+1}}\frac{F(z)}{z^{k+1}}\,dz\right|_v\leq (H_v+1)^k\max_{|z|_v=\frac{1}{H_v+1}}|F(z)|_v.
\end{equation*}
By calculations similar to \eqref{est}, for $|z|_v=\frac{1}{H_v+1}$ we have
\begin{equation*}
|F(z)|_v=\left|\left(\prod_{j=1}^r z^{i_jd_j}f(1/z)^{i_j} \right)^{1/p}\right|_v<\prod_{j=1}^r \left(1+\frac{|z|_vH_{v}}{1-|z|_v} \right)^{i_j/p}\leq 2^{\delta/p}.
\end{equation*}
So 
\begin{equation}
\label{ab}
|a_k|_v<2^{\delta/p}(H_v+1)^k, \quad k\in \mathbb{N}.
\end{equation}
By \eqref{Pseries2}, each entry in $A$ is $\zeta a_k$ for some $p$-th root of unity $\zeta$ and some $k\leq \frac{\delta}{p}$.  Now by \eqref{ad} and \eqref{ab}, 
\begin{equation*}
H(A)\leq \frac{1}{p}\left(2p^2\prod_{v\in T_\infty}(H_v+1)^{[L_v:\mathbb{Q}_v]/[L:\mathbb{Q}]}\right)^{\delta/p}=\frac{(2p^2B)^{\delta/p}}{p},
\end{equation*}
where $H(A)$ is the absolute multiplicative height of $A$ as a point of $\mathbb{P}^{(\delta+1)em(\delta)-1}$.

Note that $A$ is a matrix over $K(\zeta_p)$ of rank $m(\delta)-N$.  We now apply an appropriate version of Siegel's lemma, due to Bombieri and Vaaler \cite[Cor.\ 2.9.9]{BG}.
\begin{lemma}[Bombieri-Vaaler ]
Let $A$ be an $m\times n$ matrix of rank $r$ with entries in a number field $L$.  Then the $L$-vector space of solutions of $A\mathbf{x}=0$ has a basis $\mathbf{x}_1,\ldots,\mathbf{x}_{n-r}\in\co_{K}^n$ such that
\begin{equation*}
\prod_{i=1}^{n-r}H(\mathbf{x}_i)\leq |D_L|^{\frac{n-r}{2[L:\mathbb{Q}]}}\left(\sqrt{n}H(A)\right)^r.
\end{equation*}
\end{lemma}

Thus, there exists a basis $\mathbf{b}_1,\ldots,\mathbf{b}_N\in \co_{\mathbb{K}(\zeta_p)}^{m(\delta)}$ of the nullspace of $A$ with 
\begin{equation}
\label{bh}
H(\mathbf{b}_i)\leq \prod_{j=1}^NH(\mathbf{b}_j)\leq \left|D_{K(\zeta_p)}\right|^{\frac{N}{2[K(\zeta_p):\mathbb{Q}]}}\left(\frac{\sqrt{m(\delta)}(2p^2B)^{\delta/p}}{p}\right)^{m(\delta)-N},
\end{equation}
for $i=1,\ldots, N$.  Let $g_j$, $j=1,\ldots, N$, be the polynomials
\begin{equation*}
g_j(x_0,\ldots, x_r)=\sum_{\bi=(i_0,\ldots,i_r) \in I(\delta)}c_{\bi,j} x_0^{i_0}x_1^{i_1}\cdots x_r^{i_r},
\end{equation*}
where $c_{\bi,j}$,  $\bi \in I(\delta)$, is the solution to \eqref{sys} corresponding to $\mathbf{b}_j$.  It follows from our discussion above that $g_1(x,y_1,\ldots,y_r),\ldots, g_N(x,y_1,\ldots,y_r)$ form a basis for $L(\delta (D_\infty-E)-E)$. Furthermore, \eqref{gh} now follows from \eqref{bh}.

We now prove \eqref{gv}.  Let $j\in\{1,\ldots,N\}$.  Let $M$ be a monomial (in $x,y_{1,j_{1,v}},\ldots,y_{r,j_{r,v}}$) of $g_j\left(x,y_{1,j_{1,v}},\ldots,y_{r,j_{r,v}}\right)$.  Since $\ord_x g_j\left(x,y_{1,j_{1,v}}(x),\ldots,y_{r,j_{r,v}}(x)\right)<0$, it will suffice to consider only the principal part of the Puiseux expansion of $M(x)$.  So let $M_{<0}$ denote the principal part of the Puiseux expansion of $M$.  In the notation of \eqref{Pseries2},
\begin{equation*}
M_{<0}(x)=\zeta x^{\sum_{l=0}^r i_ld_l/p}\sum_{k=\lfloor\sum_{l=0}^r i_ld_l/p\rfloor+1 }^\infty a_kx^{-k}
\end{equation*}
for some $(i_0,\ldots,i_r)\in I(\delta)$ with $\sum_{l=0}^r \frac{i_ld_l}{p}\leq \frac{\delta}{p}$.
 Let $v\in T'$ be an archimedean absolute value.  Using \eqref{ab}, an easy estimate gives
\begin{equation*}
|M_{<0}(\alpha)|_v<|\alpha|_v^{-1/p}\frac{1}{2^{1/p}}(2H_v+2)^{(\delta+1)/p}\sum_{k=0}^\infty \left(\frac{H_v+1}{|\alpha|_v}\right)^k= \frac{(2H_v+2)^{(\delta+1)/p}}{2^{1/p}|\alpha|_v^{1/p}\left(1- \frac{H_v+1}{|\alpha|_v}\right)}.
\end{equation*}
Since $g_j$ is a sum of at most $m(\delta)$ such monomials, it follows that 
\begin{equation*}
|g_j\left(\alpha,y_{1,j_{1,v}}(\alpha),\ldots,y_{r,j_{r,v}}(\alpha)\right)|_v=|g_j\left(\alpha,\beta_1,\ldots,\beta_r\right)|_v<\frac{m(\delta)|g_j|_v(2H_v+2)^{(\delta+1)/p}}{2^{1/p}|\alpha|_v^{1/p}\left(1- \frac{H_v+1}{|\alpha|_v}\right)}.
\end{equation*}
Now suppose that $v\in T'$ is nonarchimedean.  Then $|\alpha|_v>\frac{1}{|p|_v^2}$ and by an argument similar to the above, using \eqref{ad},
\begin{equation*}
|g_j\left(\alpha,\beta_1,\ldots,\beta_r\right)|_v\leq \frac{|g_j|_v}{|\alpha|_v^{1/p}|p|_v^{2(\delta+1)/p-1}}
\end{equation*}
\end{proof}

We now finish the proof of Theorem \ref{mainth}.  Let 
\begin{equation*}
\delta= \left\lceil \frac{p^{r-1}((p-1)(d-1)-2)+t+2}{p^{r-1}-t}\right\rceil.
\end{equation*}
Then it is easily checked that $\delta\geq (p-1)(d-1)-1$ and $(\delta+1)\deg E<m(\delta)$.  Let $g_1,\ldots, g_N\in K[\zeta_p][x_0,\ldots, x_r]$ be the polynomials from Lemma \ref{mainl}.  Then $g_1(x,y_1,\ldots,y_r),\ldots, g_N(x,y_1,\ldots,y_r)$ form a basis for $L(\delta (D_\infty-E)-E)$.  The quantity $\delta$ was chosen precisely so that 
\begin{equation*}
\deg \delta (D_\infty-E)-E\geq \delta(p^{r-1}-t)-t\geq 2g(\tilde{C}). 
\end{equation*}
It is a standard fact then that the linear system $|\delta (D_\infty-E)-E|$ is base-point free.  Thus, for some $i$, $g_i(\alpha,\beta_1,\ldots,\beta_r)\neq 0$.

Let $v\in T$ be an archimedean absolute value.  Then $|f_j(\alpha)|_v\leq |f_j|_v(|\alpha|_v+1)^{d_j}$.  It follows that each monomial $M$ of $g_i(\alpha,\beta_1,\ldots,\beta_r)$ satisfies $|M(\alpha,\beta_1,\ldots,\beta_r)|_v\leq (H_v(|\alpha|_v+1))^{\delta/p}$.  Thus 
\begin{equation*}
|g_i(\alpha,\beta_1,\ldots,\beta_r)|_v\leq |g_i|_vm(\delta)(H_v(|\alpha|_v+1))^{\delta/p}.
\end{equation*}
In particular, if $|\alpha|_v\leq 2H_v+1$,
\begin{equation*}
|g_i(\alpha,\beta_1,\ldots,\beta_r)|_v\leq |g_i|_vm(\delta)(H_v(2H_v+2))^{\delta/p}<\frac{|g_i|_vm(\delta)2^{(\delta+1)/p}(H_v+1)^{(2\delta+1)/p}}{\max \left\{1,|\alpha|_v^{1/p}\right\}}.
\end{equation*}
If $|\alpha|_v> 2H_v+1$ then $v\in T'$, and so by \eqref{gv} and the fact that $\frac{H_v+1}{2H_v+1}\leq \frac{2}{3}$,
\begin{equation*}
\left|g_i\left(\alpha,\beta_1,\ldots,\beta_r\right)\right|_v\leq \frac{3m(\delta)|g_i|_v(2H_v+2)^{(\delta+1)/p}}{2^{1/p}|\alpha|_v^{1/p}}.
\end{equation*}

Thus, in all cases, for $v\in T_\infty$,
\begin{equation*}
|g_i(\alpha,\beta_1,\ldots,\beta_r)|_v<\frac{4|g_i|_vm(\delta)2^{(\delta+1)/p}(H_v+1)^{(2\delta+1)/p}}{\max \left\{1,|\alpha|_v^{1/p}\right\}}.
\end{equation*}

Now suppose that $v\in T$ is nonarchimedean.  Then since $f_i\in \co_K[x]$,
\begin{equation*}
|g_i(\alpha,\beta_1,\ldots,\beta_r)|_v\leq |g_i|_v\max\left\{1,|\alpha|_v^{\delta/p}\right\}.
\end{equation*}
If $|\alpha|_v>\frac{1}{|p|_v^2}$, then we have \eqref{gv}.
Thus,
\begin{equation*}
|g_i(\alpha,\beta_1,\ldots,\beta_r)|_v\leq \frac{|g_i|_v}{\max\left\{1,|\alpha|_v^{1/p}\right\}|p|_v^{2(\delta+1)/p}}.
\end{equation*}
For $v\not\in T$, since $\alpha\in \co_{L,T}$, $|g_i(\alpha,\beta_1,\ldots,\beta_r)|_v\leq |g_i|_v$.

So, since $g_i(\alpha,\beta_1,\ldots,\beta_r)\neq 0$, by the product formula,
\begin{equation*}
\prod_{v\in M_L} \left\|g_i(\alpha,\beta_1,\ldots,\beta_r)\right\|_v=1\leq \frac{4H(g_i)m(\delta)\left(p^{2(\delta+1)}2^{\delta+1}B^{2\delta+1}\right)^{1/p}}{H(\alpha)^{1/p}}.
\end{equation*}
Therefore,
\begin{equation*}
H(\alpha)\leq H(g_i)^p m(\delta)^pp^{2(\delta+1)}2^{\delta+2p+1}B^{2\delta+1}.
\end{equation*}
Using \eqref{gh} and $0<m(\delta)-N\leq (\delta+1)t$, this proves \eqref{h1}.

We now consider the case where $p\mid d_i$ for $i=1,\ldots, r$.  The proof in this case is similar to the above, so we will only state, without proof, the important differences.  In this case, we have
\begin{align*}
m(\delta)&=p^{r-1}\left(p\left\lfloor\frac{\delta}{p}\right\rfloor+1-\frac{(p-1)(d-2)}{2}\right), \quad p\left\lfloor\frac{\delta}{p}\right\rfloor \geq (p-1)(d-2)-1,\\
g(\tilde{C})&=p^{r-1}\left(\frac{(p-1)(d-2)}{2}-1\right)+1.
\end{align*}
Let $(x)_\infty=D_\infty$.  Then $\deg D_\infty=p^r$ and $D_\infty$ is a sum of exactly $p^r$ distinct points of $\tilde{C}$.  Furthermore, $\mathcal{M}(\delta)$ is a basis for $L\left(\left\lfloor\frac{\delta}{p}\right\rfloor D_\infty\right)$.  Note that in the present case, in the Puiseux expansions of the functions $y_i(x)$, all of the exponents are integers.  This changes some of the calculations, particularly in Lemma \ref{mainl}.  For instance, one consequence is that the matrix $A$ in the proof of that lemma can be taken to have dimensions $\left(\left\lfloor \frac{\delta}{p}\right\rfloor +1\right)\deg E\times m(\delta)$.  Clearly, it is more convenient to work with the quantity $\delta'=\left\lfloor\frac{\delta}{p}\right\rfloor$.  Let $\mathcal{M}'(\delta')=\mathcal{M}(p\delta')$ and $m'(\delta')=m(p\delta')$.  Let
\begin{equation*}
\delta'=\left\lceil \frac{p^{r-1}((p-1)(d-2)-2)+t+2}{p^r-t}\right\rceil.
\end{equation*}
This $\delta'$ is chosen precisely so that the linear system $|\delta'(D_\infty-E)-E|$ will be basepoint free.  Now calculations similar to the first case proved above give the result.
\end{proof}
We now complete the proofs of Theorems \ref{sr} and \ref{wth}.
\begin{proof}[Proof of Theorem \ref{sr}]
We apply Theorem \ref{mainth} to $g_1,\ldots,g_r$, and $\alpha\in \co_{K',S}$ satisfying $a\beta^p=\prod_{i=1}^rf_i(\alpha)$ for some $\beta\in\co_{K',S}$.  Let $S'$ be the union of $S_\infty$ and the set of places $v$ of $K'$ for which $|aD_f|_v\neq 1$.  Let $L=K'(\sqrt[p]{g_1(\alpha)},\ldots,\sqrt[p]{g_r(\alpha)},\zeta_p)$.

Suppose first that either $\epsilon=0$ or $\epsilon=1$.  In either case, it is easily seen that for all $i$, $(g_i(\alpha))=\mathfrak{a}_i^p$ for some (fractional) ideal $\mathfrak{a}_i$ of $\co_{K',S'}$.  Then the same proof as Lemma~\ref{ldeg} shows that $[L:K']\leq p^{\rk_p \Cl(\co_{K',S'})+\rk \co_{K',S'}^*+1}$.  We have $\rk_p \Cl(\co_{K',S'})\leq \rk_p \Cl(\co_{K'})$ and $\rk \co_{K',S'}^*\leq |S_\infty|+\omega_{K'}(aD_f)-1$.  Let $T$ be the set of places of $L$ lying above places of $S$.  Then 
\begin{equation*}
t=|T|\leq |S|[L:K']\leq |S|p^{\omega_{K'}(aD_f)+|S_\infty|+\rk_p \Cl(\co_{K'})}.
\end{equation*}
By Theorem \ref{mainth}, $H(\alpha)$ is effectively bounded if $t<p^{r-\epsilon}$.  Explicitly, we obtain the bound \eqref{HB} (using the trivial estimate $B\leq 2\prod_{i=1}^r H(g_i)$).

Now suppose that $\epsilon=2$.  After reindexing, we can assume that $p$ doesn't divide $\deg g_r$.  In this case, it is not necessarily true that $(g_i(\alpha))=\mathfrak{a}_i^p$ for some fractional ideal $\mathfrak{a}_i$ of $\co_{K',S'}$ (this would be true if we enlarged $S'$ to contain all of $S$).  However, for $i=1,\ldots, r$, it is easy to check that $\left(\frac{g_i^{\deg g_r}(\alpha)}{g_r^{\deg g_i}(\alpha)}\right)=\mathfrak{a}_i^p$ for some (fractional) ideal $\mathfrak{a}_i$ of $\co_{K',S'}$.  Note that, since $p\nmid\deg g_r$,
\begin{equation*}
L=K'\left(\sqrt[p]{\frac{g_{1}^{\deg g_r}(\alpha)}{g_r^{\deg g_1}(\alpha)}},\ldots,\sqrt[p]{\frac{g_{r-1}^{\deg g_r}(\alpha)}{g_r^{\deg g_{r-1}}(\alpha)}},\sqrt[p]{g_r(\alpha)},\zeta_p\right).
\end{equation*}
Thus, similar to the $\epsilon=0,1$ cases, we have $[L:K']\leq p^{\rk_p \Cl(\co_{K',S'})+\rk \co_{K',S'}^*+2}$ and 
\begin{equation*}
t=|T|\leq |S|[L:K']\leq |S|p^{\omega_{K'}(aD_f)+|S_\infty|+\rk_p \Cl(\co_{K'})+1}.
\end{equation*}
By Theorem \ref{mainth}, $H(\alpha)$ is effectively bounded if $t<p^{r-1}$.  Explicitly, we obtain the bound \eqref{HB}.
\end{proof}

\begin{proof}[Proof of Theorem \ref{wth}]
We apply Theorem~\ref{mainth} to $f_1,\ldots, f_r$ with $K=K'=\mathbb{Q}$ and $S=\{\infty\}$, the archimedean prime of $\mathbb{Q}$.  Suppose that $p\nmid \deg f_i$ for some $i$.  Let $\alpha\in \mathbb{Z}$.  Let $\omega=\omega\left(\prod_{i=1}^r \pfree(f_i(\alpha))\right)$.  Suppose $\omega \leq r-2$.  Let $q_1,\ldots, q_\omega$ be the primes dividing $\prod_{i=1}^r \pfree(f_i(\alpha))$.  Then, in the notation of Theorem \ref{mainth}, we have $L\subset \mathbb{Q}\left(\sqrt[p]{q_1},\ldots,\sqrt[p]{q_\omega},\zeta_{2p}\right)$.  If $t$ is the number of archimedean places of $L$, it follows that $t\leq \frac{1}{2}p^\omega\phi(2p)\leq \frac{1}{2}p^{r-2}\phi(2p)$.  Thus, 
\begin{equation*}
t\leq
\begin{cases}
\frac{1}{2}p^{r-2}(p-1) &\text{if } p\neq 2,\\
2^{r-2} &\text{if } p=2.
\end{cases}
\end{equation*}
Note that $B=H+1\leq 2H$, $|D_{\mathbb{Q}(\zeta_p)}|=p^{p-2}$, $r\leq d$, and $[\mathbb{Q}(\zeta_p):\mathbb{Q}]=p-1$.  A computation shows that $\delta\leq 2(p-1)(d-1)-1$.  Substituting everything into Theorem \ref{mainth} gives the first part of the theorem.  The case where $p\mid \deg f_i$ for all $i$ is similar.
\end{proof}
\section{Perfect powers in products of terms in an arithmetic progression}
\label{sper}
The study of perfect powers in arithmetic progressions goes back to at least Fermat, who proved that there are no four squares in arithmetic progression.  This was generalized by Euler, who showed that a product of four terms in arithmetic progression is never a square.  In more modern times, we have, for instance, the celebrated result of Erd{\H{o}}s and Selfridge \cite{Erd} that a product of two or more consecutive integers can never be a perfect power.  For a survey of these and other related results, see \cite{Sho2} and \cite{Sho1}.

We will consider a general form of the problem.  Let $p$ be a prime number, $k$ a positive integer, and $\gamma_1,\ldots,\gamma_r$ distinct integers with $0\leq \gamma_i<k$, $i=1,\ldots, r$.  We consider the equation
\begin{equation}
\label{maineq}
by^p=(x+\gamma_1 d)\cdots(x+\gamma_rd), \quad (x,d)=1,
\end{equation}
with $b,d,x,y\in \mathbb{Z}$.  We have the following general result.

\begin{theorem}
\label{par}
The set of solutions to the equation
\begin{equation*}
by^p=(x+\gamma_1d)\cdots(x+\gamma_rd), \quad (x,d)=1,
\end{equation*}
in $b,d,x,y\in \mathbb{Z}$ with
\begin{equation}
\label{maincond}
\omega((k-1)!b)+1+\log_p \phi(2p)\left(\frac{p}{2}+\omega(d)\right)< r,
\end{equation} 
is finite.  In particular, each such solution satisfies
\begin{equation*}
\max\{|x|,|d|\}<(2kp^4r)^{p^{3r}r^2}.
\end{equation*}
\end{theorem}

\begin{proof}
Let $b,d,x=x_0,y=y_0\in \mathbb{Z}$ satisfy \eqref{maineq} and \eqref{maincond}.  We apply Theorem~\ref{mainth} to the polynomials $f_i=x+\gamma_i$, $i=1,\ldots, r$, at the point $\alpha=\frac{x_0}{d}$, with $K'=K=\mathbb{Q}$ and $S$ consisting of the set of primes dividing $d$ and the infinite place.  Note that $\alpha\in \co_{\mathbb{Q},S}$, as required.  Let $S'$ be the union of the infinite place, the set of primes dividing $b$, and the set of primes less than $k$.  Then \eqref{maineq} easily implies that for all $i$, $(df_i(\alpha))=\mathfrak{a}_i^p$ for some ideal $\mathfrak{a}_i$ of $\co_{\mathbb{Q},S'}$.  Let $\omega_0=\omega((k-1)!b)$, $L=\mathbb{Q}(\sqrt[p]{f_1(\alpha)},\ldots,\sqrt[p]{f_r(\alpha)},\zeta_p)$, and $L'=\mathbb{Q}(\sqrt[p]{q_1},\ldots,\sqrt[p]{q_{\omega_0}},\sqrt[p]{d},\zeta_{2p})$, where $q_1,\ldots,q_{\omega_0}$ are the distinct primes dividing $(k-1)!b$.  Then the same proof as Lemma \ref{ldeg} (even easier in this case) shows that $L\subset L'$.
Let $T$ and $T'$ be the set of places of $L$ and $L'$, respectively, lying above places of $S$.  Note that $[L':\mathbb{Q}]\leq\phi(2p)p^{\omega_0+1}$ and that $L'$ is totally imaginary since it contains $\zeta_{2p}$.  Thus, $T'$ contains at most $\frac{\phi(2p)}{2}p^{\omega_0+1}$ archimedean places.  Furthermore, it's clear that each of the $\omega(d)$ finite places of $S$ ramifies to at least degree $p$ in $L'$.    Thus, $T'$ contains at most $\phi(2p)p^{\omega_0}\omega(d)$ nonarchimedean places.  So
\begin{equation*}
t=|T|\leq |T'|\leq \phi(2p)p^{\omega_0}\left(\frac{p}{2}+\omega(d)\right).
\end{equation*}
By Theorem \ref{mainth}, $H(\alpha)$ is effectively bounded if $t<p^{r-1}$.  So $H(\alpha)$ is effectively bounded if \eqref{maincond} holds.  Explicitly, substituting appropriately into Theorem~\ref{mainth} gives the bound in the theorem.
\end{proof}

Of course, it is possible to prove similar theorems over other number fields.  For example, for $p=2$, we prove a variant of Theorem \ref{par}, valid for a certain class of quadratic fields.

\begin{theorem}
\label{par2}
Every solution to the equation
\begin{equation}
\label{eq2}
by^2=(x+\gamma_1d)\cdots(x+\gamma_rd),
\end{equation}
with $b,d,x,y\in \co_L$, $L=\mathbb{Q}(\sqrt{m})$, and
\begin{equation}
\label{cond2}
\omega_L((k-1)!b)+\omega(m)+\log_2(\omega_L(d)+2)+4< r,
\end{equation} 
satisfies
\begin{equation*}
H\left(\frac{x}{d}\right)<(16kr)^{2^{3r}r^2}.
\end{equation*}
\end{theorem}
\begin{proof}
The proof is similar to the proof of Theorem \ref{par}.  Let $L=\mathbb{Q}(\sqrt{m})$ and $b,d,x=x_0,y=y_0\in \co_L$ satisfy \eqref{eq2} and \eqref{cond2}.  We apply Theorem \ref{mainth} to the polynomials $f_i=x+\gamma_i$, $i=1,\ldots, r$, at the point $\alpha=\frac{x_0}{d}$, with $K'=L$, $K=\mathbb{Q}$, and $S$ consisting of the union of the set of primes of $L$ dividing $d$ and the set of archimedean places of $L$.  Let $S'$ be the union of $S_\infty$, the set of primes of $L$ dividing $b$, and the set of primes of $L$ dividing an integer less than $k$.  Then \eqref{eq2} implies that for all $i$, $\left(\frac{f_i(\alpha)}{f_r(\alpha)}\right)=\mathfrak{a}_i^2$ for some ideal $\mathfrak{a}_i$ of $\co_{L,S'}$.  Let 
\begin{equation*}
L'=L\left(\sqrt{f_1(\alpha)},\ldots,\sqrt{f_r(\alpha)}\right)=L\left(\sqrt{\frac{f_1(\alpha)}{f_r(\alpha)}},\ldots,\sqrt{\frac{f_{r-1}(\alpha)}{f_r(\alpha)}},\sqrt{f_r(\alpha)}\right).
\end{equation*}
Then as in Lemma \ref{ldeg}, if $\zeta$ generates the roots of unity in $L$, we have $[L':L]\leq [L(\sqrt{\zeta}):L]2^{\rk_2 \Cl(\co_{L,S'})+\rk \co_{L,S'}^*+1}$.  Note that $[L(\sqrt{\zeta}):L]= 2$ and $\rk \co_{L,S'}^*=|S'|-1\leq \omega_L((k-1)!b)+1$.  From genus theory, we have an exact formula for $\rk_2 \Cl(\co_L)$, depending on the primes dividing the discriminant of $L$.  We will only use the inequality $\rk_2 \Cl(\co_L) \leq \omega(m)$.  So $[L':L]\leq 2^{\omega_L((k-1)!b)+\omega(m)+3}$.
Let $T$ be the set of places of $L'$ lying above places of $S$.  Since $|S|\leq \omega_L(d)+2$, 
\begin{equation*}
t=|T|\leq (\omega_L(d)+2)2^{\omega_L((k-1)!b)+\omega(m)+3}.
\end{equation*}
By Theorem \ref{mainth}, $H(\alpha)$ is effectively bounded if $t<2^{r-1}$.  Thus, $H(\alpha)$ is effectively bounded if \eqref{cond2} holds.  Substituting appropriately into Theorem \ref{mainth} gives the bound in the theorem.
\end{proof}


For an integer $n>1$, we let $P(n)$ denote the largest prime divisor of $n$.  We set $P(1)=1$.  As usual, we let $\pi(n)$ denote the number of primes up to (and including) $n$.  We now take up the task of using our effective bounds to completely solve some cases of Theorem \ref{par} for $p=2$.   Namely, we prove:
\begin{theorem}
\label{mainr}
Let $8\leq k \leq 17$ and $\gamma_1,\ldots,\gamma_r$ be distinct integers with $0\leq \gamma_i<k$, $i=1,\ldots, r$, and $\gamma_1=0$.  Let $\epsilon_{d,k}=0$ if the squarefree part of $d$ has no prime divisor larger than $k-1$ and $\epsilon_{d,k}=1$ otherwise.  Then every solution to
\begin{equation}
\label{p2}
by^2=(x+\gamma_1d)\cdots(x+\gamma_rd), \quad (x,d)=1, P(b)<k,
\end{equation}
with $b,d,x,y$ positive integers and
\begin{equation}
\label{ieq}
\omega(d)<2^{r-\pi(k-1)-\epsilon_{d,k}}-2
\end{equation} 
satisfies one of the following:
\begin{align*}
&d=1, &&x\in\{1, 2, 3, 4, 5, 6, 7, 8, 9, 10, 11, 12, 13, 14, 15, 16, 18, 20, 21, 22, 24, 25, 26,\\
&&&\qquad 27, 28, 30, 32, 33, 35, 36, 39, 40, 42, 44, 45, 48, 49, 50, 52, 54, 55, 56, 60,\\
&&&\qquad 63, 64, 65, 66, 70, 72, 75, 77, 84, 88, 90, 96, 98, 117, 120\},\\
&d=2, &&x\in\{1, 3, 5, 7, 9, 11, 13, 15, 21, 25, 33\},\\
&d=3, &&x\in\{1, 2, 4, 5, 7, 8, 10, 11, 13, 14, 16, 20, 22, 25, 32\},\\
&d=4, &&x\in\{1, 3, 5, 7, 9, 11, 13, 21\},\\
&d=5, &&x\in\{1, 2, 3, 4, 6, 7, 8, 9, 11, 12, 13, 14, 16, 18, 21, 24, 28, 39\},\\
&d=7, &&x\in\{1, 2, 3, 4, 5, 6, 8, 9, 11, 12, 13, 15, 16, 18, 20, 26, 30, 44\},\\
&d=8, &&x\in\{1, 9\},\\
&d=9, &&x\in\{4, 8\},\\
&d=11, &&x\in\{3, 4, 6, 10, 15, 26, 48\},\\
&d=13, &&x\in\{1, 7\},\\
&d=17,  &&x\in\{5, 22\},\\
&d=19, &&x=4,\\
&d=23, &&x=16.
\end{align*}

\end{theorem}
In particular, we solve \eqref{p2} in positive integers for the values of $\psi=k-r$, $k$, and $\omega(d)$ given in Table \ref{Tab}.  Note that the theorem for $8\leq k\leq 17$ is implied by the special cases of the theorem where $k$ is prime ($k=11,13,17$).  Thus, it suffices to list only these values of $k$ in Table \ref{Tab}.  We can also prove results for $k<8$, but these turn out to be covered by previous results, which we now discuss.

Suppose that $\psi=k-r=0$.  Equation \eqref{p2} has infinitely many solutions for $k=2,3$, $b=1$ and $k=4$, $b=6$.  Erd{\H{o}}s and Selfridge \cite{Erd} proved that \eqref{p2} has no solution with $d=1$, as long as the right-hand side of \eqref{p2} is divisible by a prime greater than or equal to $k$.  As mentioned earlier, Euler proved the nonexistence of solutions in the case $k=4$, $b=1$.  When $k=5$, Obl{\'a}th \cite{Obl} proved that \eqref{p2} does not hold if $b=1$ and Mukhopadhyay and Shorey \cite{MS2} handled the general case $P(b)<k$.  When $6\leq k\leq 11$, $P(b)\leq 5$, Bennett, Bruin, Gy{\H{o}}ry, and Hajdu \cite{BBGH} showed that the only solution to \eqref{p2} is $k=6, d=1, x=1$.  When $8\leq k \leq 100$, $d>1$, Hirata-Kohno, Laishram, Shorey, and Tijdeman \cite{HLSS} showed that \eqref{p2} does not hold except possibly in a small number of exceptional cases.  These remaining exceptional cases were handled by Tengely \cite{Ten}.  Thus, in short, we have nothing new to add in the case $\psi=0$.

Suppose that $\psi=1$.  Then Saradha and Shorey \cite{SS} showed that \eqref{p2} with $d=1$, $b=1$, $k\geq 3$ has only the solutions $\frac{6!}{5}=(12)^2, \frac{10!}{7}=(720)^2$, and that \eqref{p2} with $d=1$, $k\geq 4$, $x>k^2$ has only the solution with $k=4$, $x=24$.  In another paper, Saradha and Shorey \cite{SS2} showed that \eqref{p2} does not hold with $\omega(d)=1$, $k\geq 30$.  Mukhopadhyay and Shorey \cite{MS3} improved this to $\omega(d)=1$, $k\geq 9$, and Laishram, Shorey, and Tengely \cite{LST} improved this to $\omega(d)=1$, $k\geq 7$.  Shorey \cite{Sho} has also proved the case $\omega(d)=1$, $b=1$, $6\leq k \leq 8$.  Theorem \ref{mainr} gives some new results when $\psi=1$.  For instance, that \eqref{p2} does not hold for $2\leq \omega(d)\leq 5$, $9\leq k\leq 17$.

Suppose that $\psi=2$.  If $k\geq 4$, $d=1$, $b=1$, Mukhopadhyay and Shorey \cite{MS4} give the finitely many solutions to \eqref{p2}.  Under the assumption that the right-hand side of \eqref{p2} is divisible by a prime larger than $k$, Laishram and Shorey \cite{SS3} completely solved \eqref{p2} for $k\geq 5$, $d=1$.  Furthermore, they also showed \cite{SS3} that \eqref{p2} does not hold with $k\geq 15$, $\omega(d)=1$.  Using this result and Theorem~\ref{mainr}, we obtain the improvement that \eqref{p2} does not hold for $k\geq 9$, $\omega(d)=1$.

For $3\leq \psi\leq 7$,  Mukhopadhyay and Shorey \cite{MS4} completely solved \eqref{p2} for $d=1, b=1$, $k\geq \psi+2$, and $x>k^2$.  There do not seem to be other previous general results for $\psi>2$.
\begin{table}
\center
\caption{Table of values $\psi=k-r$, $k$ prime, and $\omega(d)$ for which Theorem \ref{mainr} gives a complete solution in positive integers to \eqref{p2}.\newline}
\begin{tabular}{|c|c|c|c|}
\hline
$\psi$ & $(k=11)$ & $(k=13)$ & $(k=17)$\\
&$\omega(d)\leq$ & $\omega(d)\leq$ & $\omega(d)\leq$\\
\hline
0& 61 & 125 & 1021\\
1& 29 & 61  & 509\\
2& 13 & 29  & 253\\
3& 5  & 13  & 125\\
4& 1  & 5   & 61\\
5& 0  & 1   & 29\\
6&    & 0   & 13\\
7&    &     & 5\\
8&    &     & 1\\
9&    &     & 0\\
\hline
\end{tabular}
\label{Tab}
\end{table}


As an immediate corollary to Theorem \ref{mainr} we obtain:
\begin{corollary}
Let $8\leq k \leq 17$.    Let $\epsilon_{d,k}$ be as in Theorem {\rm\ref{mainr}}.  Let $x$ and $d$ be positive integers not among the explicit values given in Theorem  {\rm\ref{mainr}}.  Then there are at least
\begin{equation*}
k-\pi(k-1)-\epsilon_{d,k}-\left\lfloor\log_2 \left(\omega(d)+2\right)\right\rfloor
\end{equation*}
prime divisors larger than $k-1$ dividing
\begin{equation*}
x(x+d)\cdots (x+(k-1)d)
\end{equation*}
to an odd power.
\end{corollary}

We now discuss the proof of Theorem \ref{mainr}.  To compute an effective bound for $x$ and $d$, we follow the proof of Theorem \ref{par}.  We note that there is a slight improvement in \eqref{ieq} as compared to \eqref{maincond} at $p=2$.  Firstly, an examination of the proof of Theorem \ref{par} yields an expression with the more precise quantity $\epsilon_{d,k}$.  Secondly, since we are only considering positive solutions $x$ and $d$, we can avoid adjoining $\sqrt{-1}$ to the relevant field in the proof of Theorem \ref{par}, giving a minor improvement over \eqref{maincond}.  Now, using the proof of Theorem \ref{par}, we easily calculate using Theorem \ref{mainth} that any solution in Theorem~\ref{mainr} satisfies $\max\{\log x,\log d \}<10^{14}$.  Obviously, it is infeasible to naively search for solutions within this bound.  The key point that will allow us to effectively search this space is that for any solution to $by^2=(x+\gamma_1d)\cdots(x+\gamma_rd)$, we can find many distinct elliptic curves of the form $b'Y^2=(X+\gamma_{i_1})\cdots(X+\gamma_{i_4})$, where $i_1,\ldots,i_4\in\{1,\ldots,r\}$, $b'\in \mathbb{Z}$, and $X=\frac{x}{d}$ is the $X$-coordinate of a rational point on the curve.  Then using the group structure on such curves combined with congruence conditions, we will arrive at an efficient way of searching the potential solution space.  The details are as follows.

Let $k,\gamma_1,\ldots,\gamma_r$ be as in Theorem \ref{mainr}.  Let $b,d,x,y$ be positive integers satisfying \eqref{p2} and \eqref{ieq}.  Since $P(b)<k$ and $(x,d)=1$, it follows that for each $i$, 
\begin{equation*}
x+\gamma_id=a_iz_i^2,
\end{equation*}
for some integer $z_i$ and some positive square-free integer $a_i$ satisfying $P(a_i)<k$.  Theorem \ref{mainr} is vacuous for $8\leq k \leq 13$ if $r\leq 5$ and for $14\leq k \leq 17$ if $r\leq 7$.  So we can assume that $r\geq 6$ if $8\leq k \leq 13$ and $r\geq 8$ if $14\leq k \leq 17$.  In either case, we can find six terms $x+\gamma_{i_j}d$, $j=1,\ldots, 6$, such that $P(a_{i_j})\leq 11$, $j=1,\ldots, 6$.  If $14\leq k\leq 17$, this follows since $r\geq 8$ and $13$ can divide at most two terms $x+\gamma_id$.  From the solution $b,d,x,y$, we obtain a point with $X=\frac{x}{d}$ on each of the $\binom{6}{4}=15$ elliptic curves
\begin{equation}
\label{ells}
E_J:\left(\prod_{j\in J}a_{i_j}\right)Y^2=\prod_{j\in J} X+\gamma_{i_j},
\end{equation}
where $J$ ranges over four-element subsets of $\{1,\ldots,6\}$.  We note that after translating the smallest $\gamma_{i_j}$ to zero, we have $\binom{16}{5}$ possibilities for $\gamma_{i_1},\ldots,\gamma_{i_6}$.  At first glance, for each choice of $\gamma_{i_1},\ldots, \gamma_{i_6}$, there are $2^{30}$ possibilities for the sextuplet $(a_{i_1},\ldots, a_{i_6})$.  However, this number is drastically reduced by the trivial observation that if $p|a_i, a_j$, then $p|\gamma_i-\gamma_j$.

Our problem is therefore reduced to finding, for every possible choice of the $a_{i_j}$ and $\gamma_{i_j}$, rational points on the elliptic curves \eqref{ells} which share a common $X$-coordinate of height $h(X)<10^{14}$.  In actual practice, it is convenient to choose a Weierstrass model for the curves $E_J$ (and hence also an identity element for the group law on each $E_J$).  If $a_J=\prod_{j\in J}a_{i_j}$, we have the Weierstrass model of $E_J$,
\begin{equation*}
F_J:Y^2=X(X+a_J(\gamma_{i_2}-\gamma_{i_1})(\gamma_{i_4}-\gamma_{i_3}))(X+a_J(\gamma_{i_3}-\gamma_{i_1})(\gamma_{i_4}-\gamma_{i_2})),
\end{equation*}
and a birational map $F_J\rightarrow E_J$ given by 
\begin{equation*}
(X,Y)\mapsto \left(\frac{a_J(\gamma_{i_2}-\gamma_{i_1})(\gamma_{i_3}-\gamma_{i_1})(\gamma_{i_4}-\gamma_{i_1})}{X-a_J(\gamma_{i_2}-\gamma_{i_1})(\gamma_{i_3}-\gamma_{i_1})}-\gamma_{i_1},\frac{(\gamma_{i_2}-\gamma_{i_1})(\gamma_{i_3}-\gamma_{i_1})(\gamma_{i_4}-\gamma_{i_1})Y}{(X - a_J(\gamma_{i_2}-\gamma_{i_1})(\gamma_{i_3}-\gamma_{i_1}))^2}\right).
\end{equation*}
We consider three cases, depending on the Mordell-Weil groups of the curves $E_J$.

Case I:  One of the elliptic curves $E_J$ in \eqref{ells} can be proven to have rank $0$ (over $\mathbb{Q}$).  This is the easiest case.  In this case, we easily determine finitely many possibilities for $X$ by computing the finitely many rational torsion points of $E_J$.

Case II:  Two distinct curves $E_J$ and $E_{J'}$ in \eqref{ells} can be proven to have rank $1$ (and generators for the Mordell-Weil groups can be computed).  Let 
\begin{align*}
\mathcal{X}_J&=\{X(P)\mid P\in E_J(\mathbb{Q})\},\\
\mathcal{X}_{J'}&=\{X(P)\mid P\in E_{J'}(\mathbb{Q})\},\\
\mathcal{X}&=\mathcal{X}_J\cap\mathcal{X}_{J'}.
\end{align*}
We want to determine the set of points $X\in \mathcal{X}$ with $h(X)<10^{14}$.  Let $P_J$ and $P_{J'}$ be generators, modulo torsion, for $E_J(\mathbb{Q})$ and $E_{J'}(\mathbb{Q})$, respectively.  We find primes $p$ such that $P_J$ has small order modulo $p$ but $P_{J'}$ has relatively large order modulo $p$.  This is easily done by (in a Weierstrass model) factoring the denominators of the coordinates of $mP_{J}$ for small $m$.  Since $P_J$ has small order modulo $p$, the elements in $\mathcal{X}_J$ are restricted to a small number of congruence classes modulo $p$.  If $\opo_p(P_{J'})$ denotes the order of $P_{J'}$ modulo $p$, then we find, by looking modulo $p$, that for any torsion point $T\in E_{J'}(\mathbb{Q})$, $X(nP_{J'}+T)\in \mathcal{X}_J$ implies that $n$ lies in a small number of congruence class modulo $\opo_p(P_{J'})$.  Using the theory of canonical heights on elliptic curves, one can explicitly compute a positive integer $N$ such that $h(X(nP_{J'}+T))>10^{14}$ for any torsion point $T$ and any $n$ with $|n|\geq N$.  Now we choose primes $p_1,\ldots,p_m$ as above until we have $\LCM(\opo_{p_1}(P_{J'}),\ldots,\opo_{p_m}(P_{J'}))>N$, where $\LCM$ denotes the least common multiple.  Combining the information from each prime $p_i$, we find that we only need to check if $X(nP_{J'}+T)\in \mathcal{X}_J$, $T\in E_{J'}(\mathbb{Q})_{\tors}$, for a small number of integers $n$ with $|n|<N$.  For those integers $n$ we need to check, since for large integers $n$, computing $nP_{J'}$ is impractical, we work again modulo primes $p$, checking whether $X(nP_{J'}+T) \mod p$ is the $X$-coordinate of a point in $E_J(\mathbb{F}_p)$.  In practice, this process is very efficient, typically taking only a few seconds on a modern computer to compute the points $X\in \mathcal{X}$ with $h(X)<10^{14}$, for any two given curves $E_J$ and $E_{J'}$ with given generators $P_J$ and $P_{J'}$.

Nearly all of the possibilities encountered in Theorem \ref{mainr} are covered by Case I and Case II.  However, there are a small number of instances which do not fit into these cases.  For example, if $(\gamma_{i_1},\ldots,\gamma_{i_6})=(0, 1, 2, 10, 13, 14)$ and $(a_{i_1},\ldots,a_{i_6})=(1, 15, 14, 6, 3, 2)$, then each of the $15$ curves in \eqref{ells} has rank at least $2$.  All of these remaining exceptional cases are handled in Case III.

Case III:  There exists a curve $E_J$ in \eqref{ells} which can be proven to have rank $2$, with computable generators $P_1$ and $P_2$, modulo torsion, and sets of primes $\mathcal{Q}_1,\ldots,\mathcal{Q}_t$ such that for any $i$, $|\mathcal{Q}_i|\geq 2$, $(\opo_q(P_1),\opo_q(P_2))$ is the same for all $q\in \mathcal{Q}_i$ (denote the common value by $(\opo_{\mathcal{Q}_i}(P_1),\opo_{\mathcal{Q}_i}(P_2))$), and 
\begin{equation}
\min\{\LCM(\opo_{\mathcal{Q}_1}(P_1),\ldots,\opo_{\mathcal{Q}_t}(P_1)),\LCM(\opo_{\mathcal{Q}_1}(P_2),\ldots,\opo_{\mathcal{Q}_t}(P_2))\}
\end{equation}
is sufficiently large.  In practice, in every case we were able to choose $t=2$ and $|\mathcal{Q}_1|,|\mathcal{Q}_2|\geq 3$.  Similar to before, we now use congruence conditions to restrict the linear combinations of $P_1$ and $P_2$ which must be examined.  For a point $P\in E_J(\mathbb{Q})$ and a prime $q$, let $P_q$ denote the image of $P$ in $E_J(\mathbb{F}_q)$.  Let $i\in \{1,\ldots, t\}$.  For $q\in \mathcal{Q}_i$ and $T\in E_J(\mathbb{Q})_{\tors}$, we compute the set
\begin{equation*}
I_{q,T}=\left\{(m,n)\in \mathbb{F}_{\opo_{\mathcal{Q}_i}(P_1)}\times \mathbb{F}_{\opo_{\mathcal{Q}_i}(P_2)}\mid X(mP_{1q}+nP_{2q}+T_q)\in\bigcap_{\substack{J'\subset\{1,\ldots,6\}\\|J'|=4}}X(E_{J'}(\mathbb{F}_q))\right\}.
\end{equation*}
Then we compute the set  $I_{\mathcal{Q}_i,T}=\cap_{q\in \mathcal{Q}_i}I_{q,T}$.  It follows that for any torsion point $T$ on $E_J$, we need only look at points $mP_1+nP_2+T$ on $E_J$ such that $(m \mod \opo_{\mathcal{Q}_i}(P_1),n \mod \opo_{\mathcal{Q}_i}(P_2))\in I_{\mathcal{Q}_i,T}$.  Using canonical heights, we compute positive integers $M$ and $N$ such that $h(X(mP_1+nP_2+T))>10^{14}$, if $|m|>M$,  $|n|>N$, $T\in E_J(\mathbb{Q})_{\tors}$.  Finally, we piece together the congruence information from the sets $I_{\mathcal{Q}_i,T}$ to determine a relatively small number of integers $m$ and $n$ for which we look at the points $X(mP_1+nP_2+T)$ to determine if they give a solution to \eqref{p2}.

This completes a rough description of the computation performed to prove Theorem \ref{mainr}.  The necessary Mordell-Weil groups were computed using Cremona's mwrank program (through Sage \cite{sage}) while the other computations were done using PARI/GP \cite{PARI2}.
\bibliography{Runge}
\end{document}